\documentclass[11pt,oneside,a4paper]{article}

\usepackage{ifxetex}
\ifxetex
	\usepackage{amsthm,amsmath}
	\usepackage{fontspec}
	\usepackage[]{unicode-math}
	\usepackage{csquotes}
	\usepackage{polyglossia}
	\setdefaultlanguage{english}

\else
	\RequirePackage[utf8]{inputenc}
	\RequirePackage[T1]{fontenc}
	
	\RequirePackage[final]{microtype}
	
	\RequirePackage{amsfonts}
	\RequirePackage{amsthm,amsmath}
	\RequirePackage{amssymb}
	
	\usepackage{csquotes}
	\usepackage[english]{babel}
\fi

\usepackage[backend=biber, style=numeric, sorting=anyt, autopunct=true, maxnames=10, maxalphanames=10, natbib=true, arxiv=abs, doi=true, url=false, eprint=true]{biblatex}
\usepackage[margin=1.5in]{geometry}
\usepackage[title,titletoc]{appendix}

\usepackage{graphicx}

\usepackage{float}
\floatstyle{plain}
\restylefloat{figure}

\usepackage{subcaption}
\usepackage[ruled,section]{algorithm}
\usepackage{algpseudocode}

\usepackage[plain]{fancyref}

\newcommand*{\fancyrefthmlabelprefix}{thm}
\Frefformat{plain}{\fancyrefthmlabelprefix}{%
	Theorem\fancyrefdefaultspacing#1%
}%
\frefformat{plain}{\fancyrefthmlabelprefix}{%
	theorem\fancyrefdefaultspacing#1%
}%

\newcommand*{\fancyreflemmalabelprefix}{lemma}
\Frefformat{plain}{\fancyreflemmalabelprefix}{%
	Lemma\fancyrefdefaultspacing#1%
}%
\frefformat{plain}{\fancyreflemmalabelprefix}{%
	lemma\fancyrefdefaultspacing#1%
}%

\newcommand*{\fancyrefclaimlabelprefix}{claim}
\Frefformat{plain}{\fancyrefclaimlabelprefix}{%
	Claim\fancyrefdefaultspacing#1%
}%
\frefformat{plain}{\fancyrefclaimlabelprefix}{%
	claim\fancyrefdefaultspacing#1%
}%

\newcommand*{\fancyrefdeflabelprefix}{def}
\Frefformat{plain}{\fancyrefdeflabelprefix}{%
	Definition\fancyrefdefaultspacing#1%
}%
\frefformat{plain}{\fancyrefdeflabelprefix}{%
	definition\fancyrefdefaultspacing#1%
}%

\newcommand*{\fancyrefconjlabelprefix}{conj}
\Frefformat{plain}{\fancyrefconjlabelprefix}{%
	Conjecture\fancyrefdefaultspacing#1%
}%
\frefformat{plain}{\fancyrefconjlabelprefix}{%
	conjecture\fancyrefdefaultspacing#1%
}%

\newcommand*{\fancyrefcaselabelprefix}{case}
\Frefformat{plain}{\fancyrefcaselabelprefix}{Section\fancyrefdefaultspacing#1}%
\frefformat{plain}{\fancyrefcaselabelprefix}{section\fancyrefdefaultspacing#1}%

\newcommand*{\fancyrefremlabelprefix}{rem}
\Frefformat{plain}{\fancyrefremlabelprefix}{Remark\fancyrefdefaultspacing#1}%
\frefformat{plain}{\fancyrefremlabelprefix}{remark\fancyrefdefaultspacing#1}%

\newcommand*{\fancyrefitemlabelprefix}{item}
\Frefformat{plain}{\fancyrefitemlabelprefix}{Item\fancyrefdefaultspacing#1}%
\frefformat{plain}{\fancyrefitemlabelprefix}{item\fancyrefdefaultspacing#1}%

\newcommand*{\fancyrefproplabelprefix}{prop}
\Frefformat{plain}{\fancyrefproplabelprefix}{Proposition\fancyrefdefaultspacing#1}%
\frefformat{plain}{\fancyrefproplabelprefix}{proposition\fancyrefdefaultspacing#1}%

\newcommand*{\fancyrefphaselabelprefix}{phase}
\Frefformat{plain}{\fancyrefphaselabelprefix}{Phase\fancyrefdefaultspacing#1}%
\frefformat{plain}{\fancyrefphaselabelprefix}{phase\fancyrefdefaultspacing#1}%

\newcommand*{\fancyrefineqlabelprefix}{ineq}
\Frefformat{plain}{\fancyrefineqlabelprefix}{Inequality\fancyrefdefaultspacing(#1)}%
\frefformat{plain}{\fancyrefineqlabelprefix}{inequality\fancyrefdefaultspacing(#1)}%

\newcommand*{\fancyreftablelabelprefix}{table}
\Frefformat{plain}{\fancyreftablelabelprefix}{Table\fancyrefdefaultspacing#1}%
\frefformat{plain}{\fancyreftablelabelprefix}{table\fancyrefdefaultspacing#1}%

\newcommand*{\fancyrefappendixlabelprefix}{appendix}
\Frefformat{plain}{\fancyrefappendixlabelprefix}{Appendix\fancyrefdefaultspacing#1}%
\frefformat{plain}{\fancyrefappendixlabelprefix}{appendix\fancyrefdefaultspacing#1}%

\newcommand*{\fancyrefalgolabelprefix}{algo}
\Frefformat{plain}{\fancyrefalgolabelprefix}{Algorithm\fancyrefdefaultspacing#1}%
\frefformat{plain}{\fancyrefalgolabelprefix}{algorithm\fancyrefdefaultspacing#1}%

\newcommand*{\fancyrefsteplabelprefix}{step}
\Frefformat{plain}{\fancyrefsteplabelprefix}{Step\fancyrefdefaultspacing#1}%
\frefformat{plain}{\fancyrefsteplabelprefix}{step\fancyrefdefaultspacing#1}%

\usepackage{float}
\usepackage{enumerate}
\usepackage[usenames,dvipsnames,svgnames,table]{xcolor}
\colorlet{highlightnew}{black}

\usepackage{booktabs}
\usepackage{multirow}

\usepackage{hyperref} %unicode,hidelinks

\hypersetup %http://en.wikibooks.org/wiki/LaTeX/Hyperlinks
{
   unicode=true,          % non-Latin characters in Acrobat’s bookmarks
   pdfstartview={XYZ null null 1},    % fits the width of the page to the window
   pdftitle={Mobile vs.\ point guards},
   pdfauthor={Ervin Gy\H{o}ri, Tam\'as R\'obert Mezei},
   pdfkeywords={art gallery problem, orthogonal polygon, mobile guard, sliding cameras}, % list of keywords
   colorlinks=false,        % false: boxed links; true: colored links
}

\usepackage{tikz}
\usepackage{tkz-berge}
\usetikzlibrary{calc,patterns,turtle,fit,shapes}
\usepackage{pgfplots}
\pgfplotsset{compat=newest}
\def\dist{0.5}
\def\pgsize{3pt}
\def\mgsize{thick}
\tikzset{ every picture/.style={line cap=butt,turtle/distance=\dist,thick} } %scale=0.5, transform canvas={scale=0.5} %line cap=rect
\ifxetex
	
\else
	
\fi

\newcommand{\vpat}{dotted}
\newcommand{\hpat}{dotted}

\colorlet{inside}{white!80!black}
\colorlet{insidelight}{white!90!black}

\newtheorem{theorem}{Theorem}
\newtheorem{lemma}[theorem]{Lemma}
\newtheorem{claim}[theorem]{Claim}
\newtheorem{corollary}[theorem]{Corollary}
\newtheorem{proposition}[theorem]{Proposition}
\theoremstyle{definition}
\newtheorem{definition}[theorem]{Definition}

\makeatletter
\newcommand{\neutralize}[1]{\expandafter\let\csname c@#1\endcsname\count@}
\makeatother

\newenvironment{thmbis}[1]
{\renewcommand{\thetheorem}{\ref*{#1}$^\prime$}%
	\neutralize{thm}\phantomsection%
	\begin{theorem}}
	{\end{theorem}}

\newcommand{\rela}{\leftrightarrow}
\newcommand{\covers}{\rightarrow}

\usepackage{dsfont}

\newcounter{case}[section]
\newcommand{\casemark}[5]{}
\makeatletter
\newcommand\case{\@startsection{case}{2}{\z@}%
	{12\p@ \@plus 6\p@ \@minus 3\p@}%
	{3\p@ \@plus 6\p@ \@minus 3\p@}%
	{\normalfont\normalsize\bfseries\boldmath Case }}

\newcounter{subcase}[case]
\newcommand{\subcasemark}[5]{}
\newcommand\subcase{\@startsection{subcase}{3}{\z@}%
	{12\p@ \@plus 6\p@ \@minus 3\p@}%
	{\p@}%
	{\normalfont\normalsize\bfseries\boldmath Case }}

\newcounter{subsubcase}[subcase]
\renewcommand{\thesubsubcase}%
{\arabic{case}.\arabic{subcase}.\arabic{subsubcase}}
\newcommand{\subsubcasemark}[5]{}
\newcommand\subsubcase{\@startsection{subsubcase}{4}{\z@}%
	{12\p@ \@plus 6\p@ \@minus 3\p@}%
	{\p@}%
	{\normalfont\normalsize\bfseries\boldmath Case }}

\newcounter{phase}[subsection]
\newcommand{\phasemark}[5]{}
\newcommand\phase{\@startsection{phase}{3}{\z@}{3.25ex \@plus1ex \@minus.2ex}{-1em}{\normalfont\normalsize\bfseries\boldmath Phase }}
\makeatother

\title{\bfseries Mobile vs.\ point guards\footnote{This version supersedes its version published in \text{Discrete \& Computational Geometry} by covering a case missing from the original proof. Phases~2~and~3 have been extended, as previously we only considered cycles in $M'$, but not circuits. Moreover, $M'_V$ is now defined analogously to $M'_H$, ie.,\ certain vertical slices are split into two pieces. This required a slight adjustment of the computations in Phase~1 and \Fref{sec:estimating}.}}

\author{
	Ervin Gy\H ori$^{1,2,}$\footnote{Research  of the authors was supported by NKFIH grant K-116769.} \\
	\texttt{gyori.ervin@renyi.mta.hu}
	\and
	Tam\'as R\'obert Mezei$^{1,2,}$\addtocounter{footnote}{-1}\footnotemark\ \textsuperscript{,}\footnote{Corresponding author} \\
	\texttt{tamasrobert.mezei@gmail.com}
	}
\date{
	\small $^{1}$Alfréd Rényi Institute of Mathematics, Hungarian Academy of Sciences, Re\'altanoda~u.~13--15, 1053 Budapest, Hungary \\
	$^{2}$Central European University, Department of Mathematics and its Applications, N\'ador~u.~9, 1051 Budapest, Hungary \\
	\vspace{12pt}
	\normalsize \today
	\vspace{-12pt}
	}

\begin{document}

\maketitle
%\linenumbers

\begin{abstract}
	We study the problem of guarding orthogonal art galleries with horizontal mobile guards (alternatively, vertical) and point guards, using ``rectangular vision''.
	We prove a sharp bound on the minimum number of point guards required to cover the gallery in terms of the minimum number of vertical mobile guards and the minimum number of horizontal mobile guards required to cover the gallery.
	Furthermore, we show that the latter two numbers can be computed in linear time.
\end{abstract}

\section{Introduction}
The number of mobile and point guards required to control the interior of a general or an orthogonal polygon (without holes) has been well-studied as a function of the number of vertices of the polygon (in the introduction we assume the reader is familiar with the concept of mobile guards, point guards, etc., but all of these notion are defined precisely in \Fref{sec:defs}). Kahn, Klawe, and Kleitman in 1980~\cite{MR699771}, and a few years later Győri~\cite{MR844048}, and O'Rourke~\cite{ORourke} proved that $\lfloor n/4\rfloor$ point guards are sufficient and sometimes necessary to cover the interior of an orthogonal polygon of $n$ vertices. Aggarwal proved in his thesis~\cite{Ag} that any $n$-vertex orthogonal polygon can be covered by at most $\lfloor\frac{3n+4}{16}\rfloor$ mobile guards, and a strengthening of this result has been shown in~\cite{GyM2016}. These estimates are also shown to be sharp as extremal results. These theorems imply that --- from an extremal point of view --- only $4/3$ times as many point guards as mobile guards are needed. However, the ratio of these optima has not been studied.
% In \cite{ORourke}*{Table~3.1} it is shown that as a function of the number of vertices of an art gallery which is bounded by a general or an orthogonal polygon, at most $4/3$'s as many point guards as mobile guards are needed.

\medskip

The main goal of this paper is to explore the ratio between the numbers of mobile guards and points guards required to control an orthogonal polygon  without holes. At first, this appears to be hopeless, as \Fref{fig:comb} shows a comb, which can be guarded by one mobile guard (whose patrol is shown by a dotted horizontal line). However, to cover the comb using point guards, one has to be placed for each tooth, so ten point guards are needed (marked by solid disks). Combs with arbitrarily high number of teeth clearly demonstrate that the minimum number of points guards required to control an orthogonal polygon cannot be bounded by the minimum size of a mobile guard system covering the comb.
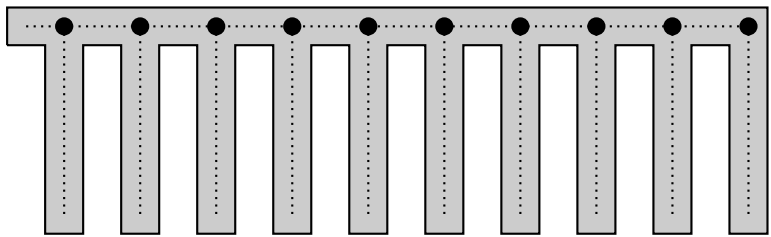
\begin{figure}[ht]
	\centering
	\begin{tikzpicture}
		\def\n {10}
		\def\l {5}

		\begin{scope}

			\filldraw [fill=inside][turtle={home,forward,right}]
			\foreach \i in {1,...,\n}
				{[turtle={fd,fd}]}

				[turtle=right,fd]
			\foreach \i in {1,...,\n}
				{
					\foreach \j in {1,...,\l}
						{[turtle=fd]}
						[turtle=rt,fd,rt]
					\foreach \j in {1,...,\l}
						{[turtle=fd]}
						[turtle=lt,fd,lt]
				};

		\end{scope}

		\begin{scope}[\mgsize]
			\draw[\hpat] (\dist*0.5,\dist*0.5) -- (\dist*2*\n-\dist*0.5,\dist*0.5);

			\foreach \x in {1,...,\n}
				{
					\draw[\vpat] (\dist*2*\x-\dist*0.5,\dist*0.5) -- (\dist*2*\x-\dist*0.5,-\dist*\l+\dist*0.5);
					\filldraw[black] (\dist*2*\x-\dist*0.5,\dist*0.5) circle ( \pgsize );
				}
		\end{scope}
	\end{tikzpicture}
	\caption{A comb with 10 teeth}\label{fig:comb}
\end{figure}

\medskip

%Notice, however, that the minimum size of a vertical mobile guard system of the comb in \Fref{fig:comb} is the same as the minimum size of a point guard system of it.

{\color{highlightnew} In this paper, we study point and mobile guards that are equipped with rectangular vision, or $r$-vision for short: two points are visible to each other if their axis-parallel bounding rectangle is contained in the gallery. The results of \cite{MR844048,ORourke,GyM2016} show that the worst case bounds on the number of point- and mobile guard required to control an $n$-vertex orthogonal polygon do not increase if line of sight vision is restricted to $r$-vision.}

\medskip

{\color{highlightnew} Even though the point guard problem in orthogonal polygons is \textsc{NP}-hard for line of sight vision \cite{MR1330860}, the problem becomes polynomial for $r$-vision \cite{WM07}. The $\tilde{\mathcal{O}}(n^{17})$ time complexity is brought down by \citet{MR3598396} to a linear running time for thin orthogonal polygons. (An orthogonal polygon is thin if for any point $x$ in the gallery there exists a vertex $v$ on the orthogonal polygon to which everything seen by $x$ via $r$-vision is $r$-visible.) Furthermore, a linear time 3-approximation algorithm for the point guard problem with $r$-vision in orthogonal polygons has been developed by \citet{LWZ12}.}

\medskip

\textcolor{highlightnew}{\citet{KM11}) defined and studied the notion of ``horizontal sliding cameras'', which is a horizontal line segment $h\subset D$ inside the gallery, which sees a point $x\in D$ in the gallery if there is a point $y\in h$ on the line segment such that $\overline{xy}\perp h$. For a maximal horizontal line segment, the area covered by $h$ as a horizontal mobile $r$-guard (guard with rectangular vision) and as a horizontal sliding camera are identical up to a 0-measure subset (see \Fref{lemma:technical}).}

\medskip

The main result of our paper, \Fref{thm:main}, shows that a constant factor times the sum of the minimum sizes of a horizontal and a vertical mobile $r$-guard system can be used to estimate the minimum size of a point $r$-guard system. It is surprising to have such a result given that this ratio cannot be bounded if the region may contain holes.

\medskip

Take, for example, \Fref{fig:holes}, which generally contains $3k^2+4k+1$ square holes (in the figure $k=4$). The regions covered by line of sight vision by the black dots are pairwise disjoint, because the distance between adjacent square holes is less than half of the length of a square hole's side.
Therefore no two of the black dots can be covered by one point guard,
%(even if it uses line of sight vision)
so at least $k^2$ point guards are necessary to control gallery. However, $2k+2$ horizontal mobile guards can easily cover the polygon, and the same holds for vertical mobile guards.
\begin{figure}
	\centering
	\begin{tikzpicture}[scale=0.3]
		\def \n {4}
		\filldraw[fill=inside] (-0.25,-0.25) rectangle (4*\n+2.75,4*\n+2.75);

		\foreach \x in {1,...,\n}
			{
				\foreach \y in {1,...,\n}
					{
						\filldraw[fill=white] (4*\x-2,4*\y-2) rectangle ++(-1.5,-1.5);
						\filldraw[fill=white] (4*\x-2,4*\y) rectangle ++(-1.5,-1.5);
						\filldraw[fill=white] (4*\x,4*\y-2) rectangle ++(-1.5,-1.5);

						\draw[black,fill=black] (4*\x-0.75,4*\y-0.75) circle (0.5ex);
					}
			}

		\foreach \x in {1,...,\n}
			{
				\filldraw[fill=white] (4*\x-2,4*\n+2) rectangle ++(-1.5,-1.5);
				\filldraw[fill=white] (4*\x,4*\n+2) rectangle ++(-1.5,-1.5);
				\filldraw[fill=white] (4*\n+2,4*\x-2) rectangle ++(-1.5,-1.5);
				\filldraw[fill=white] (4*\n+2,4*\x) rectangle ++(-1.5,-1.5);
			}
		\filldraw[fill=white] (4*\n+2,4*\n+2) rectangle ++(-1.5,-1.5);
	\end{tikzpicture}
	\caption{A polygon with holes --- unlimited ratio.}\label{fig:holes}
\end{figure}
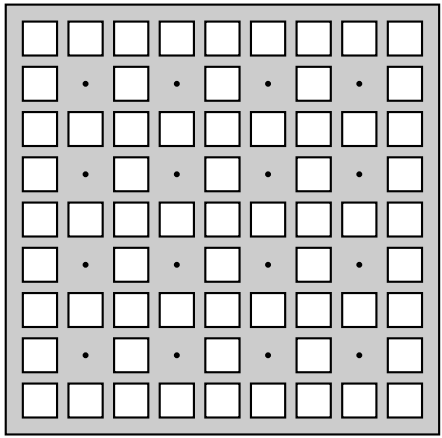

\medskip

In the last section of the paper, we show that a minimum size horizontal mobile $r$-guard system can be found in linear time (\Fref{thm:mgalg}). This improves the result in~\cite{KM11}, where it is shown that this problem can be solved in polynomial time.

%\Fref{fig:comb} shows a comb, which can be guarded by 1 mobile guard (whose patrol is shown in \textbf{\textcolor{red}{red}}). Thus $m=1$ and $m_H=1$, but to cover the comb using only vertical mobile guards (shown in \textbf{\textcolor{blue}{blue}}), one has to be placed in each tooth, so $m_V=10$. Similarly, $p=10$, as one point guard (shown in \textbf{\textcolor{green}{green}}) needs to placed in each tooth. Combs with arbitrarily high number of teeth clearly demonstrate that there exists no univariate function of either $m$, $m_V$, or $m_H$ which bounds $p$ from above for every orthogonal polygon. \figcap{comb}{A comb with 10 teeth}

%\paragraph{Our results.} We study the problem of $r$-guarding orthogonal art galleries with vertical mobile guards (alternatively, horizontal) and point guards. We prove a sharp bound on the minimum number of point guards required to cover the gallery in terms of the minimum number of vertical mobile guards and the minimum number of horizontal mobile guards required to cover the gallery (\Fref{thm:main}).

\section{Definitions and preliminaries}\label{sec:defs}

Our universe for the study of art galleries is the plane $\mathbb{R}^2$. A \textbf{polygon} is defined by a cyclically ordered list of pairwise distinct vertices in the plane. It is drawn by joining each successive pair of vertices %(including the pair formed by the first and last vertices)
on the list by line segments, that only intersect in vertices of the polygon. The last requirement ensures that the closed domain bounded by the polygon is simply connected (to emphasize this, such polygons are often referred to as simple polygons in the literature). An \textbf{orthogonal polygon} is a polygon such that its line segments are alternatingly parallel to one of the axes of $\mathbb{R}^2$. Consequently, it is simply connected, and its angles are $\frac12\pi$ (convex) or $\frac32\pi$ (reflex).

\medskip

A \textbf{rectilinear domain} is a closed region of the plane ($\mathbb{R}^2$) whose boundary is an orthogonal polygon, i.e., a closed polygon without self-intersection, so that each segment is parallel to one of the two axes. A \textbf{rectilinear domain with holes} is a rectilinear domain with pairwise disjoint simple rectilinear domain holes. Its boundary is referred to as an \textbf{orthogonal polygon with holes}.

\medskip

The definitions imply that number of vertices of an orthogonal polygon (even with holes) is even. We denote the number of vertices of the polygon by $n(P)$, and define $n(D)=n(P)$, where $D$ is the domain bounded by $P$. Conversely, we write $P=\partial D$. We want to emphasize that in our problems not just the walls, but also the interior of the gallery must be covered. In the proofs of the theorems, therefore, we are working on rectilinear domains, not orthogonal polygons, even though one defines the other uniquely, and vice versa.

\medskip

%Since mostly simply connected domains are discussed in this thesis, we omit the adjective simple in front of orthogonal polygons and rectilinear domains.
Whenever results about objects that are allowed to have holes are mentioned, it is explicitly stated.

\medskip

To avoid confusion, we state that throughout this part, \textbf{vertices} and \textbf{sides} refer to subsets of an orthogonal polygon or a rectilinear domain; whereas any \textbf{graph} will be defined on a set of \textbf{nodes}, of which some pairs are joined by some \textbf{edges}. Given a graph $G$, the edge set $E(G)$ is a subset of the 2-element subsets of the vertices $V(G)$.

\medskip

\begin{table}
	\color{highlightnew}
	\centering
	\bgroup%
	\def\arraystretch{1.5}
	\small
	\begin{tabular}{ c  c  c }
		\textbf{Name}      & \textbf{Notation}  & \textbf{Meaning}                                                               \\ \midrule\midrule
		Orthogonal polygon & $P$                & A simple polygon made up of horiz.\ and vert.\ segments                        \\ \midrule
		Rectilinear domain & $D$                & A bounded region of $\mathbb{R}^2$ s.t.\ $\partial D$ is an orthogonal polygon \\ \midrule
		Side               &                    & A maximal horizontal or vertical segment of $P$ or $\partial D$                \\ \midrule
		Vertex             &                    & A non-empty intersection of two distinct sides                                 \\ \midrule
		Convex hull        & $\mathrm{Conv}(X)$ & The smallest convex set containing $X\subset\mathbb{R}^2$                      \\ \midrule
		Pixel              & $\cap e$           & The intersection of the elements of $e$                                        \\ \midrule
		Centroid           & $c(X)$             & The arithmetic mean position of $X\subset \mathbb{R}^2$                        \\ \bottomrule
	\end{tabular}
	\egroup%

	\bigskip

	\caption{Notation used in the paper}\label{table:notation}
\end{table}

Unless otherwise noted, we adhere to the same terminology in the subject of art galleries as O'Rourke~\cite{ORourke}.
However, for technical reasons, sometimes we need to assume extra conditions over what is traditionally assumed. In \Fref{lemma:technical}, we prove that we may, without restricting the problem, require the assumptions typeset in \emph{italics} in the following definitions.

\medskip

Two points $x,y$ in a domain $D$ have \textbf{line of sight vision}, \textbf{unrestricted vision}, or simply just \textbf{vision} of each other if the line segment spanned by $x$ and $y$ is contained in $D$.

\medskip

A \textbf{point guard} in an art gallery $D$ is a point $y\in D$. It has vision of a point $x\in D$ if the line segment $\overline{xy}$ is a subset of $D$. The term ``stationary guard'' refers to the same meaning, and is used mostly in contrast with ``mobile guards''.

\medskip

A \textbf{mobile guard} is a line segment $L\subset D$. A point $x\in D$ is seen by the guard if there is a point $y\in L$ which has vision of $x$. Intuitively, a mobile guard is a point guard patrolling the line segment $L$.

\medskip

The points \textbf{covered by a guard} is just another name for the set of points of $D$ that are seen by the guard. A \textbf{system of guards} is a set of guards in $D$ which cover $D$, i.e., for any point $x\in D$, there is a guard in the system covering $x$.

\medskip

Two points $x,y$ in a rectilinear domain $D$ have \textbf{\boldmath $r$-vision} of each other (alternatively, $x$ is $r$-visible from $y$) if there exists an axis-aligned \emph{non-degenerate} % amúgy nem kell, ha úgyis átírjuk a polygont
rectangle in $D$ which contains both $x$ and $y$. This vision is natural to use in orthogonal art galleries instead of the more powerful line of sight vision. For example, $r$-vision is invariant on the transformation depicted on \Fref{fig:mgcond1}.

\medskip

A \textbf{point \boldmath $r$-guard} is a point $y\in D$, \emph{such that the two maximal axis-parallel line segments in $D$ containing $y$ do not intersect vertices of $D$}. A set of point guards \hbox{\textbf{\boldmath $r$-cover}} $D$ if any point $x\in D$ is $r$-visible from a member of the set. Such a set is called a \textbf{point \boldmath $r$-guard system}.

\medskip

A \textbf{vertical mobile \boldmath $r$-guard} is a vertical line segment in $D$, \emph{such that the maximal line segment in $D$ containing it does not intersect vertices of~$D$}. \textbf{Horizontal} mobile guards are defined analogously. A \textbf{mobile \boldmath $r$-guard} is either a vertical or a horizontal mobile $r$-guard. A mobile $r$-guard \hbox{\textbf{\boldmath  $r$-covers}} any point $x\in D$ for which there exists a point $y$ on its line segment such that $x$ is $r$-visible from $y$.

\begin{lemma}\label{lemma:technical}
	Any rectilinear domain $D$ can be transformed into another rectilinear domain $D'$ so that the point guard $r$-cover, and the vertical/horizontal mobile guard $r$-cover problems in $D$, without the restrictions typeset in italics, are equivalent to the respective problems, as per our definitions (i.e., with the restrictions), in $D'$.
\end{lemma}
\begin{proof}
	%	Degenerate rectangles only contribute a 0-measure set to the points visible from fixed point of $D$. Furthermore, if a set of guards cover $D$ via $r$-vision (as defined here) except a 0-measure part of it, then actually they cover $D$ entirely; the proof is trivial. Therefore, we may prohibit vision by degenerate rectangles without loss of generality.
	\begin{figure}[ht]
		\centering
		\begin{tikzpicture}
			\def\dist{0.5}
			\begin{scope}[yscale=1.5,shift={(0*\dist,0.33*\dist)},rotate=-90]
				\fill[color=inside]  (0,7*\dist) rectangle (2*\dist,8*\dist) rectangle (\dist,4.5*\dist) rectangle (0,6*\dist) rectangle (\dist,2*\dist) rectangle (2*\dist,3*\dist) rectangle (\dist,-1.5*\dist) rectangle (0,0) rectangle (\dist,-3*\dist);

				\draw (0,0) -- ++(\dist,0) -- ++(0,2*\dist) -- ++(-\dist,0);
				\draw (2*\dist,3*\dist) -- ++(-1*\dist,0) -- ++(0,1.5*\dist) -- ++(\dist,0);
				\draw (0,6*\dist) -- ++(\dist,0) -- ++(0,\dist) -- ++(-\dist,0);
				\draw (0,8*\dist) -- ++(2*\dist,0);
				\draw (0,-3*\dist) -- ++(\dist,0) -- ++(0,1.5*\dist) -- ++(\dist,0);
			\end{scope}

			\draw [very thick,decorate,decoration={coil,aspect=0},->] (8.5*\dist,-\dist) -- (10.5*\dist,-\dist);

			\begin{scope}[shift={(14*\dist,0.5*\dist)}, rotate=-90]
				\fill[color=inside]  (0,7*\dist) rectangle (3*\dist,8*\dist) rectangle (\dist,4.5*\dist) rectangle (0,6*\dist);
				\fill[color=inside] (0,5*\dist) rectangle (2*\dist,2*\dist);
				\fill[color=inside] (3*\dist,3*\dist) rectangle (\dist,-1.5*\dist);
				\fill[color=inside] (0,0) rectangle (2*\dist,-3*\dist);

				\draw (0,0) -- ++(\dist,0) -- ++(0,2*\dist) -- ++(-\dist,0);
				\draw (3*\dist,3*\dist) -- ++(-1*\dist,0) -- ++(0,1.5*\dist) -- ++(\dist,0);
				\draw (0,6*\dist) -- ++(\dist,0) -- ++(0,\dist) -- ++(-\dist,0);
				\draw (0,8*\dist) -- ++(3*\dist,0);
				\draw (0,-3*\dist) -- ++(2*\dist,0) -- ++(0,1.5*\dist) -- ++(\dist,0);
			\end{scope}
		\end{tikzpicture}
		\caption{After this transformation, those mobile guards whose maximal containing line segment does not intersect vertices of the rectilinear domain, are just as powerful as mobile guards that are not restricted in such a way.}\label{fig:mgcond1}
	\end{figure}
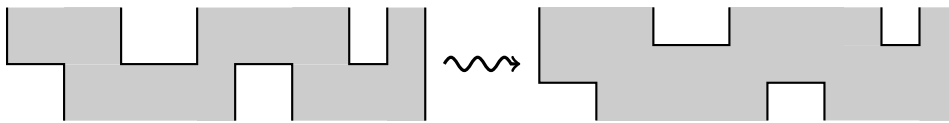
	Let $\varepsilon$ be the minimal distance between any two horizontal line segments of $\partial D$. The transformation depicted in \Fref{fig:mgcond1} in $D$ takes a maximal horizontal line segment $L$ in $D$ which is touched from both above and below by the exterior of $D$, and maps $D$ to \[D'=D\bigcup \left(L+\overline{(0,-\varepsilon/4)(0,\varepsilon/4)}\right),\]
	where addition is taken in the Minkowski sense.
	There is a trivial correspondence between the point and mobile guards of $D$ and $D'$ such that taking this correspondence guard-wise transforms a guarding system of $D$ (guards without the restrictions) into a guarding system of $D'$ (guards with the restrictions), and vice versa.

	\medskip

	After performing this operation at every vertical and horizontal occurrence, we get a rectilinear domain $D''$, in which any vertical or horizontal line segment is contained in a non-degenerate rectangle in $D''$. Therefore, degenerate vision between any two points implies non-degenerate vision between the pair. Furthermore, the line segment of any mobile guard can be translated slightly along its normal (at least in one direction) while staying inside $D''$, and this clearly does not change the set of points $r$-covered by the guard. Similarly, we can perturb the position of a point guard without changing the set of points of $D''$ it $r$-covers.
\end{proof}

\begin{theorem}\label{thm:main}
	Given a rectilinear domain $D$ let $m_V$ be the minimum size of a vertical mobile $r$-guard system of $D$, let $m_H$ be defined analogously for horizontal mobile $r$-guard systems, and finally let $p$ be the minimum size of a point $r$-guard system of $D$. Then
	\[ \left\lfloor\frac{4(m_V+m_H-1)}{3}\right\rfloor\ge p. \]
\end{theorem}

\medskip

Observe, that the magical $4:3$ ratio highlighted by \citet[Section~3.1]{ORourke} appears between the minimum number of (horizontal plus vertical) mobile and point guards required to control the gallery, even though the theorem does not use the number of vertices of the gallery as a parameter. %In case it is not confusing, the prefix ``$r$-'' is omitted from now on. 
Before moving onto the proof of \Fref{thm:main}, we discuss the aspects of its sharpness.

\medskip

For $m_V+m_H\le 6$, sharpness of the theorem is shown by the examples in \Fref{fig:sharpness}.
The polygon in \Fref{fig:sharpness6} can be easily generalized to one satisfying $m_V+m_H=3k+1$ and $p=4k$. For $m_V+m_H=3k+2$ and $m_V+m_H=3k+3$, we can attach 1 or 2 plus signs to the previously constructed polygons, as shown in \Fref{fig:sharpness4} and~\ref{fig:sharpness5}.
Thus \Fref{thm:main} is sharp for any fixed value of $m_V+m_H$.

\medskip

By stringing together a number of copies of the polygons in \Fref{fig:sharpness1} and~\ref{fig:sharpness3} in an L-shape (\Fref{fig:sharpness6} is a special case of this), we can construct rectilinear domains for any $(m_H,m_V)$ pair satisfying $m_V\le 2(m_H-1)$ and $m_H\le 2(m_V-1)$, such that the polygon satisfies \Fref{thm:main} sharply.
The analysis in \Fref{sec:translating} immediately yields that if $m_V=1$ or $m_H=1$, then $m_V+m_H-1$ is an upper bound for the minimum size of a point guard system (see~\Fref{prop:star}), whose sharpness is shown by combs (\Fref{fig:comb}).

%\medskip
%
%Let $m$ be the minimum size of a mixed set of vertical and horizontal mobile $r$-guard system of $D$. Furthermore, for any integer choice of $m,m_V,m_H\ge 2$, where $m\ge m_V$ and $m\ge m_H$, we can construct a polygon whose mobile guard parameters take the previous values and $\frac43\cdot(m_V+m_H-1)$ point $r$-guards are required to cover it. This is achieved by stringing together a number of copies of the polygons in \Fref{fig:sharpness1} and~\ref{fig:sharpness3} in a \emph{staircase} shape, and attaching to it zero, one, or two plus signs.

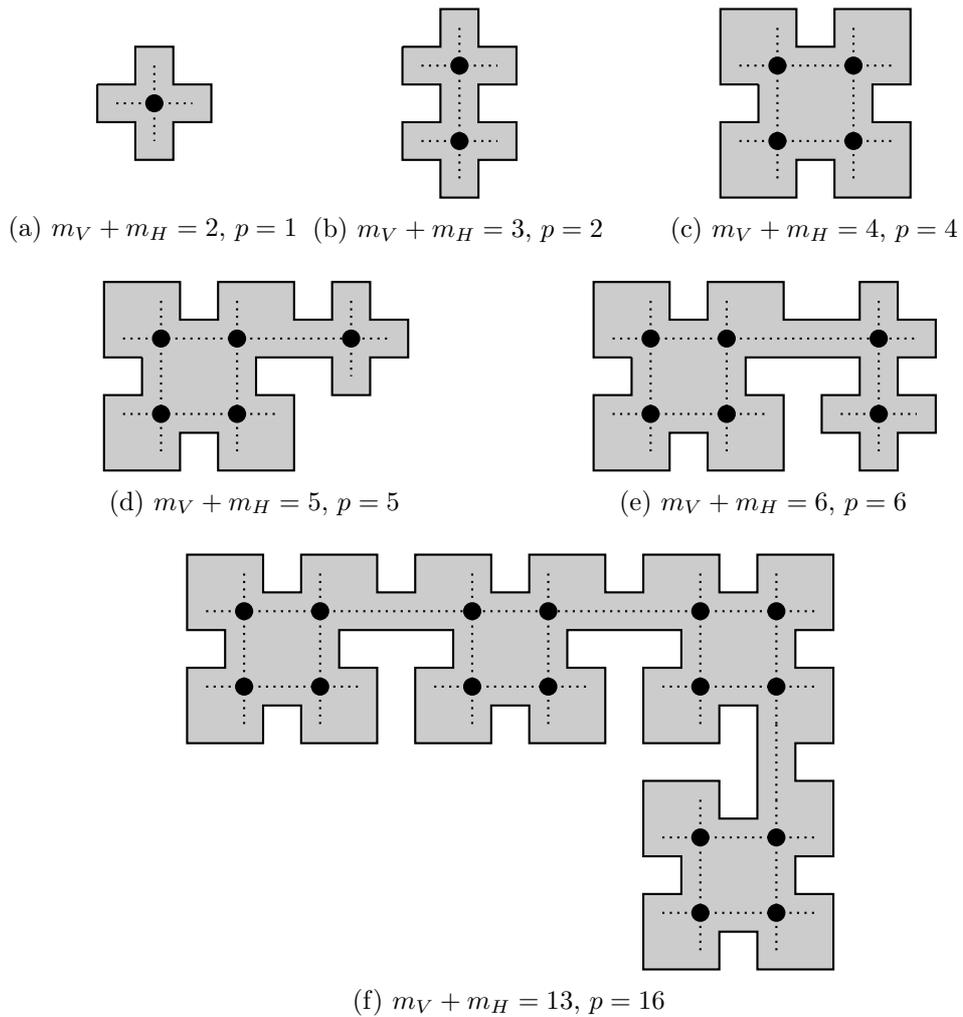
\begin{figure}
	\centering
	\begin{subfigure}{.3\textwidth}
		\centering
		\begin{tikzpicture}
			\begin{scope}

				\fill[white] (\dist,\dist) rectangle ++(\dist,\dist);
				\fill[white] (\dist,\dist*-3) rectangle ++(\dist,\dist);
				\filldraw [fill=inside][turtle=home]
				\foreach \i in {1,...,4}
					{
						[turtle={right,forward,left,forward,right,forward}]
					};
			\end{scope}

			\begin{scope}[\mgsize]
				\draw[\hpat] (\dist*0.5,\dist*-0.5) -- (\dist*2.5,\dist*-0.5);
				\draw[\vpat] (\dist*1.5,\dist*0.5) -- (\dist*1.5,\dist*-1.5);
				\filldraw[black] (\dist*1.5,\dist*-0.5) circle ( \pgsize );
			\end{scope}
		\end{tikzpicture}
		\caption{\small $m_V+m_H=2$, $p=1$}\label{fig:sharpness1}
	\end{subfigure}%
	\begin{subfigure}{.3\textwidth}
		\centering
		\begin{tikzpicture}
			\begin{scope}

				\filldraw [fill=inside][turtle=home]
				\foreach \i in {1,2}
					{
						[turtle={right,forward,left,forward,right,forward}]
					}
					[turtle={right,forward,left,left}]
				\foreach \i in {1,2,3}
					{
						[turtle={right,forward,left,forward,right,forward}]
					}
					[turtle={right,forward,left,left}]
					[turtle={right,forward,left,forward,right,forward}]
				;
			\end{scope}

			\begin{scope}[\mgsize]
				\draw[\hpat] (\dist*0.5,\dist*-0.5) -- (\dist*2.5,\dist*-0.5);
				\draw[\hpat] (\dist*0.5,\dist*-2.5) -- (\dist*2.5,\dist*-2.5);
				\draw[\vpat] (\dist*1.5,\dist*0.5) -- (\dist*1.5,\dist*-3.5);
				\filldraw[black] (\dist*1.5,\dist*-0.5) circle ( \pgsize );
				\filldraw[black] (\dist*1.5,\dist*-2.5) circle ( \pgsize );

			\end{scope}
		\end{tikzpicture}
		\caption{\small $m_V+m_H=3$, $p=2$}\label{fig:sharpness2}
	\end{subfigure}%
	\begin{subfigure}{.40\textwidth}
		\centering
		\begin{tikzpicture}
			\begin{scope}

			\end{scope}

			\begin{scope} % shiftelni
				\filldraw [fill=inside][turtle=home]
				\foreach \i in {1,2,3,4}
					{
						[turtle={left,forward,right,forward,forward,right,forward,forward,right,forward,left,forward}]
					};

			\end{scope}

			\begin{scope}[\mgsize]

				\foreach \x in {\dist*0.5,\dist*2.5}
				\draw[\vpat] (\x,\dist*1.5) -- (\x,\dist*-2.5);

				\foreach \y in {\dist*0.5,\dist*-1.5}
				\draw[\hpat] (\dist*-0.5,\y) -- (\dist*3.5,\y);

				\foreach \x in {\dist*0.5,\dist*2.5}
				\foreach \y in {\dist*0.5,\dist*-1.5}
				\filldraw[black] (\x,\y) circle ( \pgsize );

			\end{scope}
		\end{tikzpicture}
		\caption{\small $m_V+m_H=4$, $p=4$}\label{fig:sharpness3}
	\end{subfigure}

	\bigskip

	\begin{subfigure}{.5\textwidth}
		\centering
		\begin{tikzpicture}
			\begin{scope}

			\end{scope}

			\begin{scope} % shiftelni\\
				\filldraw [fill=inside][turtle=home]			[turtle={left,forward,right,forward,forward,right,forward,forward,right,forward,left,forward}]
				[turtle={left,forward,right,forward,forward,right,forward,left}]

				\foreach \i in {1,...,3}
					{
						[turtle={forward,left,forward,right,forward,right}]
					}
					[turtle={forward,left,forward,forward,left,forward}]

				\foreach \i in {1,2}
					{
						[turtle={left,forward,right,forward,forward,right,forward,forward,right,forward,left,forward}]
					};

			\end{scope}

			\begin{scope}[\mgsize]

				\foreach \x in {\dist*0.5,\dist*2.5}
				\draw[\vpat] (\x,\dist*1.5) -- (\x,\dist*-2.5);

				\draw[\hpat] (\dist*-0.5,\dist*0.5) -- (\dist*6.5,\dist*0.5);
				\draw[\vpat] (\dist*5.5,\dist*1.5) -- (\dist*5.5,\dist*-0.5);
				\filldraw[black] (\dist*5.5,\dist*0.5) circle ( \pgsize );

				\foreach \y in {\dist*-1.5}
				\draw[\hpat] (\dist*-0.5,\y) -- (\dist*3.5,\y);

				\foreach \x in {\dist*0.5,\dist*2.5}
				\foreach \y in {\dist*0.5,\dist*-1.5}
				\filldraw[black] (\x,\y) circle ( \pgsize );

			\end{scope}

		\end{tikzpicture}
		\caption{\small $m_V+m_H=5$, $p=5$}\label{fig:sharpness4}
	\end{subfigure}%
	\begin{subfigure}{.5\textwidth}
		\centering
		\begin{tikzpicture}
			\begin{scope}

			\end{scope}

			\begin{scope} % shiftelni
				\filldraw [fill=inside][turtle=home]			[turtle={left,forward,right,forward,forward,right,forward,forward,right,forward,left,forward}]
				[turtle={left,forward,right,forward,forward,right,forward,left,forward}]

				\foreach \i in {1,2}
					{
						[turtle={forward,left,forward,right,forward,right}]
					}
					[turtle={forward,left,left}]
				\foreach \i in {1,2,3}
					{
						[turtle={right,forward,left,forward,right,forward}]
					}
					[turtle={right,forward,left,left}]
					[turtle={right,forward,left,forward,forward,forward,left,forward}]

				\foreach \i in {1,2}
					{
						[turtle={left,forward,right,forward,forward,right,forward,forward,right,forward,left,forward}]
					};

			\end{scope}

			\begin{scope}[\mgsize]

				\foreach \x in {\dist*0.5,\dist*2.5}
				\draw[\vpat] (\x,\dist*1.5) -- (\x,\dist*-2.5);

				\draw[\hpat] (\dist*-0.5,\dist*0.5) -- (\dist*7.5,\dist*0.5);
				\draw[\hpat] (\dist*5.5,\dist*-1.5) -- (\dist*7.5,\dist*-1.5);
				\draw[\vpat] (\dist*6.5,\dist*1.5) -- (\dist*6.5,\dist*-2.5);
				\filldraw[black] (\dist*6.5,\dist*0.5) circle ( \pgsize );
				\filldraw[black] (\dist*6.5,\dist*-1.5) circle ( \pgsize );

				\foreach \y in {\dist*-1.5}
				\draw[\hpat] (\dist*-0.5,\y) -- (\dist*3.5,\y);

				\foreach \x in {\dist*0.5,\dist*2.5}
				\foreach \y in {\dist*0.5,\dist*-1.5}
				\filldraw[black] (\x,\y) circle ( \pgsize );

			\end{scope}
		\end{tikzpicture}
		\caption{\small $m_V+m_H=6$, $p=6$}\label{fig:sharpness5}
	\end{subfigure}

	\bigskip

	\begin{subfigure}{\textwidth}
		\centering
		\begin{tikzpicture}

			\def\n{2}
			\begin{scope}

				\filldraw [fill=inside][turtle=home]

				% felül előre
				\foreach \i in {1,...,5} %2n+1
					{
						[turtle={fd,rt,fd,fd,rt,fd,lt,fd,lt}]
					}

					[turtle={fd,rt,fd}]

				\foreach \i in {1,...,3}
					{
						[turtle={fd,rt,fd,fd,rt,fd,lt,fd,lt}]
					}

				\foreach \i in {1,2,3}
					{
						[turtle={fd,rt,fd,fd,rt,fd,fd,rt,fd,lt,fd,lt}]
					}

					[turtle={fd,fd,fd,lt,fd,lt,fd,rt,fd,fd,rt,fd,fd,rt,fd,lt,fd,lt,fd}]

				\foreach \i in {1,2}
					{
						[turtle=fd]

						\foreach \j in {1,2}
							{
								[turtle={fd,lt,fd,lt,fd,rt,fd,fd,rt,fd,fd,rt}]
							}

							[turtle={fd,lt,fd,lt,fd}]

					}

					[turtle={rt,fd}]

				;
			\end{scope}

			\begin{scope}[\mgsize]

				\draw[\hpat] (\dist*0.5,\dist*-0.5) -- (\dist*16.5,\dist*-0.5);

				\foreach \z in {0,...,\n}
					{
						\begin{scope}[shift={(\z*\dist*6,0)}]
							\foreach \x in {\dist*1.5,\dist*3.5}
							\draw[\vpat] (\x,\dist*0.5) -- (\x,\dist*-3.5);

							\draw[\hpat] (\dist*0.5,\dist*-2.5) -- (\dist*4.5,\dist*-2.5);

							\foreach \x in {\dist*1.5,\dist*3.5}
							\foreach \y in {\dist*-0.5,\dist*-2.5}
							\filldraw[black] (\x,\y) circle ( \pgsize );
						\end{scope}
					}

				\begin{scope}[shift={(\n*\dist*6,-\dist*6)}]
					\foreach \x in {\dist*1.5,\dist*3.5}
					\draw[\vpat] (\x,\dist*0.5) -- (\x,\dist*-3.5);

					\draw[\vpat] (\dist*3.5,\dist*2.5) -- (\dist*3.5,\dist*0.5);

					\foreach \y in {\dist*-0.5,\dist*-2.5}
					\draw[\hpat] (\dist*0.5,\y) -- (\dist*4.5,\y);

					\foreach \x in {\dist*1.5,\dist*3.5}
					\foreach \y in {\dist*-0.5,\dist*-2.5}
					\filldraw[black] (\x,\y) circle ( \pgsize );
				\end{scope}
			\end{scope}
		\end{tikzpicture}
		\caption{\small $m_V+m_H=13$, $p=16$}\label{fig:sharpness6}
	\end{subfigure}

	\caption{\small Vertical dotted lines: a minimum size vertical mobile guard system; \\ Horizontal dotted lines: a minimum size horizontal mobile guard system; \\ Solid disks: a minimum size point guard system.}\label{fig:sharpness}
\end{figure}

\section{Translating the problem into the language of graphs}\label{sec:translating}

For graph theoretical notation and theorems used in this chapter (say, the block decomposition of graphs), the reader is referred to~\cite{Diestel}.

\begin{definition}[Chordal bipartite or bichordal graph,~\cite{GolumbicGoss78}]
	A graph $G$ is chordal bipartite iff any cycle $C$ of $\ge 6$ vertices of $G$ has a chord (that is $E(G[C])\supsetneqq E(C)$).
\end{definition}

Let $S_V$ be the set of internally disjoint rectangles we obtain by cutting vertically at each reflex vertex of a rectilinear domain $D$. Similarly, let $S_H$ be defined analogously for horizontal cuts of $D$. We may refer to the elements of these sets as {\bfseries vertical and horizontal slices}, respectively.

\medskip

The horizontal $R$-tree $T_H$ of $D$ is equal to
\[ T_H=\left(S_H,\Big\{\{h_1,h_2\}\subseteq S_H\ :\ h_1\neq h_2,\ h_1\cap h_2\neq\emptyset\Big\}\right), \]
i.e., $T_H$ is the intersection graph of the horizontal slices of $D$. The graph $T_H$ is indeed a tree as its connectedness is trivial, and since any cut creates two internally disjoint rectilinear domains, $T_H$ is also cycle-free. We can think of $T_H$ as a sort of dual of the planar graph determined by the union of $\partial D$ and its horizontal cuts. Similarly, $T_V$ is the intersection graph of the vertical slices of $D$.

\medskip

Let $G$ be the intersection graph of $S_H$ and $S_V$, i.e.,
\[ G=\left(S_H\cup S_V,\left\{\{h,v\}\ :\ h\in S_H,\ v\in S_V,\ \mathrm{int}(h)\cap \mathrm{int}(v)\neq\emptyset\right\}\right).\]
In other words, a horizontal and a vertical slice are joined by an edge iff their interiors intersect; see \Fref{fig:pixelation}. We may also refer to $G$ as the \textbf{pixelation graph} of $D$.
Clearly, the {\bfseries set of  pixels} $\{\cap e\ |\ e\in E(G)\}$ is a cover of $D$. Let us define $c(e)$ as the centroid of $\cap e$ (the pixel determined by $e$).
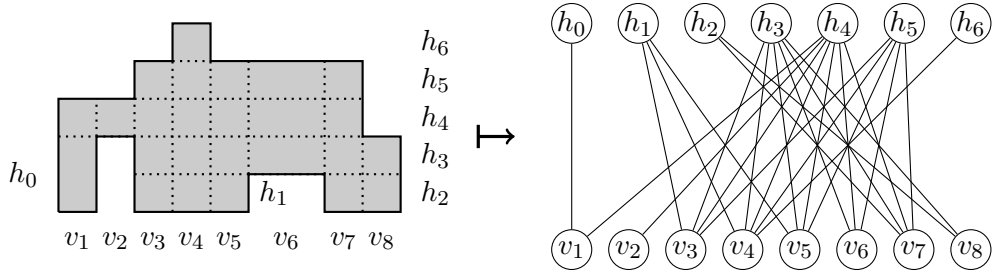
\begin{figure}[h]
	\centering
	\begin{tikzpicture}
		\begin{scope}
			\filldraw[fill=inside][turtle=home]
			[turtle=fd,fd,fd,rt,fd,fd,lt,fd,rt,fd,lt,fd,rt,fd,rt,fd,lt,fd,fd,fd,fd,rt,fd,fd,lt,fd,rt,fd,fd,rt,fd,fd,rt,fd,lt,fd,fd,lt,fd,rt,fd,fd,fd,rt,fd,fd,lt,fd,lt,fd,fd,rt,fd];

			\draw[dotted] (3*\dist,4*\dist) -- ++(\dist,0);
			\draw[dotted] (2*\dist,3*\dist) -- ++(6*\dist,0);
			\draw[dotted] (0*\dist,2*\dist) -- ++(8*\dist,0);
			\draw[dotted] (2*\dist,1*\dist) -- ++(7*\dist,0);

			\draw[dotted] (1*\dist,2*\dist) -- ++(0,1*\dist);
			\draw[dotted] (2*\dist,2*\dist) -- ++(0,1*\dist);
			\draw[dotted] (3*\dist,0*\dist) -- ++(0,4*\dist);
			\draw[dotted] (4*\dist,0*\dist) -- ++(0,4*\dist);
			\draw[dotted] (5*\dist,1*\dist) -- ++(0,3*\dist);
			\draw[dotted] (7*\dist,1*\dist) -- ++(0,3*\dist);
			\draw[dotted] (8*\dist,0*\dist) -- ++(0,2*\dist);

			\foreach \x in {1,...,5}
			\draw (\dist*\x-0.5*\dist,-0.25*\dist) node[anchor=north] {$v_\x$};

			\draw (\dist*6.5-0.5*\dist,-0.25*\dist) node[anchor=north] {$v_6$};

			\foreach \x in {7,8}
			\draw (\dist*\x+0.5*\dist,-0.25*\dist) node[anchor=north] {$v_\x$};

			\foreach \y in {2,...,6}
			\draw (9.25*\dist,\dist*\y-1.5*\dist) node[anchor=west] {$h_\y$};

			\draw (-0.25*\dist,1*\dist) node[anchor=east] {$h_0$};
			\draw (5*\dist,0.5*\dist) node[anchor=west] {$h_1$};
		\end{scope}

		\draw [very thick,|->] (11*\dist,2*\dist) -- ++(1*\dist,0);

		\begin{scope}[shift={(12*\dist,0)}]]
			\foreach \x in {0,...,6}
			\node[draw,thin,circle,inner sep=0pt,minimum size=15pt](h_\x) at (\x*7*1.5/6*\dist+1.5*\dist,5*\dist) {$h_\x$};

			\foreach \x in {1,...,8}
			\node[draw,thin,circle,inner sep=0pt,minimum size=15pt](v_\x) at
			(\x*1.5*\dist,-1*\dist) {$v_\x$};

			\foreach \v/\h in {1/0,3/1,4/1,5/1,1/4,2/4,3/3,3/4,3/5,4/3,4/4,4/5,4/6,5/3,5/4,5/5,6/3,6/4,6/5,7/2,7/3,7/4,7/5,8/2,8/3}
			\draw[thin] (v_\v) -- (h_\h); % edge[bend left=-4*\v+4*\h-6]
		\end{scope}
	\end{tikzpicture}
	\caption{A rectilinear domain and its associated pixelation graph}\label{fig:pixelation}
\end{figure}

% \vbox{\begin{claim}\label{claim:cut-tree}
% 	Both $T_H$ and $T_V$ are trees.
% \end{claim}
% \begin{proof}
% 	Connectedness of the trees follows from the connectedness of $D$. Furthermore, given an edge $e=\{h_1,h_2\}\in E(T)$, the set $D-\cap e=D- h_1\cap h_2=D-\partial h_1\cap \partial h_2$ has two components, therefore $T_H-e$ must have two components as well.
% \end{proof}}
%
% \medskip

\begin{lemma}\label{lemma:chordal}
	$G$ is a connected chordal bipartite graph.
\end{lemma}
\begin{proof}
	Connectedness of $D$ immediately yields that $G$ is connected too. Suppose $C$ is a cycle of $\ge 6$ vertices in $G$. For each node of the cycle $C$, connect the centroids of the pixels of its two incident edges with a line segment. This way we get a (not necessarily simple) orthogonal polygon $P$ in $D$.

	\medskip

	If $P$ is self-intersecting, then the vertices which are represented by the two intersecting line segments are intersecting. This clearly corresponds to a chord of $C$ in $G$.

	\medskip

	If $P$ is simple, then the number of its vertices is $|V(C)|$, thus one of them is a reflex vertex, say $c(v_1\cap h_1)$ is one. As $P$ lives in $D$, its interior is a subset of $D$ as well (here we use that $D$ is simply connected). The simpleness of $P$ also implies that the vertical line segment intersecting $c(v_1\cap h_1)$, after entering the interior of $P$ at $c(v_1\cap h_1)$, intersects $P$ at least once more when it emerges, say at $c(v_1\cap h_2)$. As this is not an intersection of the line segments corresponding to two vertices of $D$, the edge $\{v_1,h_2\}$ is a chord of~$C$.
\end{proof}

It is worth mentioning that even if $D$ is a rectilinear domain with rectilinear hole(s), $G$ may still be chordal bipartite. Take, for example, ${[0,3]}^2\setminus {(1,2)}^2$; the graph associated to it has only one cycle, which is of length 4.

\medskip

We will use the following technical claim to translate $r$-vision of points of $D$ into relations in $G$.

\begin{claim}\label{claim:rectvision}
	Let $e_1,e_2\in E(G)$, where $e_1=\{v_1,h_1\}$, $e_2=\{v_2,h_2\}$, $v_1,v_2\in S_V$, and $h_1,h_2\in S_H$. The points $p_1\in \mathrm{int}(\cap e_1)$ and $p_2\in \mathrm{int}(\cap e_2)$ have $r$-vision of each other in $D$ iff $e_1\cap e_2\neq\emptyset$ or $e_1\cup e_2$ induces a 4-cycle in $G$.
\end{claim}
\begin{proof}
	If $v_1\in e_1\cap e_2$, then $p_1,p_2\in v_1$, therefore $p_1$ and $p_2$ have $r$-vision of each other. If $h_1\in e_1\cap e_2$, the same holds. If $\{v_1,h_1,v_2,h_2\}$ induces a 4-cycle, then
	\[ \mathrm{Conv}((v_1\cap h_1)\cup (v_1\cap h_2))\subseteq v_1\subseteq D \]
	by $v_1$'s convexity.	Moreover,
	\begin{align*}
		B=   & \mathrm{Conv}((v_1\cap h_1)\cup (v_1\cap h_2))\cup \mathrm{Conv}((v_1\cap h_2)\cup (v_2\cap h_2))\cup \\
		\cup & \mathrm{Conv}((v_2\cap h_2)\cup (v_2\cap h_1))\cup \mathrm{Conv}((v_2\cap h_1)\cup (v_1\cap h_1))
	\end{align*}
	is contained in $D$. Since $D$ is simply connected, we have $\mathrm{Conv}(B)\subseteq D$, which is a rectangle containing both $p_1$ and $p_2$.

	\medskip

	In the other direction, suppose $e_1\cap e_2=\emptyset$. If $R$ is an axis-aligned rectangle which contains both $p_1$ and $p_2$, then $R$ clearly intersects the interiors of each element of $e_1\cup e_2$, which implies that $\mathrm{int}(v_2)\cap \mathrm{int}(h_1)\neq\emptyset$ and $\mathrm{int}(v_1)\cap \mathrm{int}(h_2)\neq\emptyset$. Thus $e_1\cup e_2$ induces a cycle in $G$.
\end{proof}
This easily implies the following claim.
\begin{claim}\label{claim:rectvision2}
	Two points $p_1,p_2\in D$ have $r$-vision of each other iff $\exists e_1,e_2\in E(G)$ such that $p_1\in \cap e_1$, $p_2\in \cap e_2$, and either $e_1\cap e_2\neq\emptyset$ or $e_1\cup e_2$ induces a 4-cycle in $G$.
\end{claim}

These claims motivate the following definition.

\begin{definition}[$r$-vision of edges]\label{def:rvisionedge}
	For any $e_1,e_2\in E(G)$ we say that $e_1$ and $e_2$ have $r$-vision of each other iff $e_1\cap e_2\neq\emptyset$ or there exists a $C_4$ in $G$ which contains both $e_1$ and $e_2$.
\end{definition}

Let $Z\subseteq E(G)$ be such that for any $e_0\in E(G)$ there exists an $e_1\in Z$ so that $e_1$ has $r$-vision of $e_0$.
According to \Fref{claim:rectvision2}, if we choose a point from $\mathrm{int}(\cap e_1)$ for each $e_1\in Z$, then we get a point $r$-guard system of $D$.
%Therefore in the proof of \Fref{thm:main} we only need to find a suitable $Z$, a problem which is defined conveniently in terms of graph theoretic concepts.

\medskip

Observe that any vertical mobile $r$-guard is contained in $\mathrm{int}(v)$ for some $v\in S_V$ (except $\le 2$ points of the patrol).
Extending the line segment the mobile guard patrols increases the area that it covers, therefore we may assume that this line segment intersects each element of $\{\mathrm{int}(\cap e)\ |\ v\in e\in E(G)\}$, which only depends on some $v\in S_V$.
Using \Fref{claim:rectvision2}, we conclude that the set which such a mobile guard covers with $r$-vision is exactly $\cup\{ h\in S_H\ |\ \{h,v\}\in E(G) \}$. The analogous statement holds for horizontal mobile guards as well.

\medskip

Thus, a set of vertical mobile guards of $D$ can be represented by a set $M_V\subseteq S_V$. Clearly, $M_V$ covers $D$ if and only if
\[ D = \bigcup_{v\in M_V}\left(\bigcup N_G(v)\right),\text{ which holds iff } S_H=\bigcup_{v\in M_V} N_G(v), \]
or in other words, $M_V$ dominates each element of $S_H$ in $G$. Similarly, a horizontal mobile guard system has a representative set $M_H\subseteq S_H$, which dominates $S_V$ in $G$. Equivalently, $M_H\cup M_V$ is a totally dominating set of $G$, i.e., a subset of $V(G)$ that dominates every node of $G$ (even the nodes of $M_H\cup M_V$).

\medskip

\textcolor{highlightnew}{\citet{MR2310594} studies weakly cooperative mobile guards in grids. A grid is the connected union of vertical and horizontal segments in the plane, and a mobile guard is a maximal horizontal or vertical line segment of the grid. A set of mobile guards is called weakly cooperative, if the segment of each mobile guard intersects another guard's segment. An important observation of \cite{MR2310594} is that the weakly cooperative mobile guard set problem in grids reduces to the total dominating set problem in the intersection graph of the grid. In \Fref{sec:algo}, we discuss their complexity results as well.}

\medskip

The observations about $G$ can be extended to a mixed set of vertical and horizontal mobile $r$-guards, which is represented by a set of vertices of $S\subseteq V(G)$. The set of guards is a covering system of guards of $D$ if and only if every node $V(G)\setminus S$ has neighbor in $S$, i.e., $S$ is a dominating set in $G$.
\Fref{table:translation} is the dictionary that lists the main notions of the original problem and their corresponding phrasing in the pixelation graph.

\begin{table}
	\centering
	\bgroup%
	\def\arraystretch{1.5}
	\small
	\begin{tabular}{ c  c }
		\textbf{Orthogonal polygon}      & \textbf{Pixelation graph}                               \\ \midrule\midrule
		Mobile guard                     & Vertex                                                  \\ \toprule
		Point guard                      & Edge                                                    \\ \midrule
		Simply connected                 & Chordal bipartite ($\Rightarrow$, but $\not\Leftarrow$) \\ \midrule
		$r$-vision of two points         & $e_1\cap e_2\neq\emptyset$ or $G[e_1\cup e_2]\cong C_4$ \\ \midrule
		Horiz.~mobile guard cover        & $M_H\subseteq S_H$ dominating $S_V$                     \\ \midrule
		Covering system of mobile guards & Dominating set                                          \\ \bottomrule
	\end{tabular}
	\egroup%

	\bigskip

	\caption{Translating the orthogonal art gallery problem to the pixelation graph}\label{table:translation}
\end{table}

\medskip

As promised, the following claim has a very short proof using the definitions and claims of this section.

\begin{proposition}\label{prop:star}
	If $m_V=1$ or $m_H=1$, then $p\le m_V+m_H-1$.
\end{proposition}
\begin{proof}
	Let $Z$ be the set of edges of $G$ induced by $M_H\cup M_V$. Clearly, $G[M_H\cup M_V]$ is a star, thus $|Z|=|M_H|+|M_V|-1$.

	\medskip

	We claim that $Z$ covers $E(G)$. There exist two slices, $h_1\in M_H$ and $v_1\in M_V$, which are joined by an edge to $v_0$ and $h_0$, respectively.
	Since $G[M_H\cup M_V]$ is a star, $\{v_1,h_1\}\in Z$. This edge has $r$-vision of $e_0$, as either $\{v_1,h_1\}$ intersects $e_0$, or $\{v_0,h_0,v_1,h_1\}$ induces a $C_4$ in $Z$.
\end{proof}

\medskip

Finally, we can state \Fref{thm:main} in a stronger form, conveniently via graph theoretic concepts.

\begin{thmbis}{thm:main}\label{thm:mainprime}
	Let $A_V$ be a set of internally disjoint axis-parallel rectangles of a rectilinear domain $D$, called the vertical slices. Similarly, let $A_H$ be another set with the same property, whose elements we call the horizontal slices.
	Also, suppose that for any $v\in A_V$, its top and bottom sides are a subset of $\partial D$, and for any $h\in A_H$, its left and right sides are a subset of $\partial D$.
	Furthermore, suppose that their intersection graph \[ G=\left(A_H\cup A_V,\big\{\{h,v\}\subseteq A_V\cup A_H\ :\ \mathrm{int}(v)\cap\mathrm{int}(h)\neq\emptyset\big\}\right) \]
	is connected.

	\medskip

	If $M_V\subseteq A_V$ dominates $A_H$ in $G$, and $M_H\subseteq A_H$ dominates $A_V$ in $G$, then there exists a set of edges $Z\subseteq E(G)$ such that any element of $E(G)$ is $r$-visible from some element of $Z$, and
	\[ |Z|\le \frac43\cdot(|M_V|+|M_H|-1). \]
\end{thmbis}

\medskip

Now we are ready to prove the main theorem of this paper.

\section{Proof of \texorpdfstring{\Fref{thm:mainprime}}{Theorem~\ref{thm:main}'}}
The set $A_H$ can be extended to a set $S_H$ of internally disjoint axis-parallel rectangles which completely cover $D$, and whose left and right sides are subsets of $\partial D$. Similarly, extend $A_V$ to a complete partition $S_V$ of $D$. By \Fref{lemma:chordal}, $G$ is a subgraph induced by $A_H\cup A_V$ in a chordal bipartite graph, thus $G$ is chordal bipartite as well. Let $M=G[M_V\cup M_H]$ be the subgraph induced by the dominating sets. Notice, that the bichordality of $G$ is inherited by $M$.

\medskip

{\color{highlightnew} Given a pair of subsets $A_H\subseteq S_H$ and $A_V\subseteq S_V$ such that their intersection graph $G$ is connected, join two slices $h_1,h_2\in A_H$ by an edge if there exists a $v\in A_V$ such that $\{h_1,v\},\{h_2,v\}\in E(G)$ and there does not exist $h_3\in A_H$ which is between $h_1$ and $h_2$ in the path induced by $N_G(v)$ in $T_H$. We call the constructed graph the $R$-tree on $A_H$. The definition for $A_V$ goes analogously.

	\begin{claim}\label{claim:treeprop}
		For any $h_1,h_2\in A_H$ the following statements hold:
		\begin{itemize}
			\item $N_G(h_1)$ is the vertex set of a path in the $R$-tree on $A_V$, or in other words $N_G(h_1)$ induces a path in the $R$-tree on $A_V$.
			\item $N_G(h_1)\bigcap N_G(h_2)$ is either empty, contains exactly one slice, or induces a path in the $R$-tree on $A_V$.
			\item If $G$ is 2-connected and $h_1$ is a neighbor of $h_2$ in the $R$-tree on $A_H$, then
			      \[\left|N_G(h_1)\bigcap N_G(h_2)\right|\ge 2.\]
		\end{itemize}
	\end{claim}
	\begin{proof}
		The first two statements are trivial. Suppose that $G$ is 2-connected, $h_1$ is joined to $h_2$ in the $R$-tree on $A_H$. There is a path connecting $h_1$ to $h_2$ in $G$. Every second node of this path is a vertical slice, and the neighborhoods of two vertical slices distance two apart have a common neighbor. The neighborhood of a vertical slice is path in the $R$-tree on $A_H$, so there exists a vertical slice $v_1$ such that $h_1,h_2\in N_G(v_1)$. Moreover, $G-v$ is still connected, so in the same manner we can find another vertical slice $v_2$ which is also joined to both $h_1$ and $h_2$ in $G$.
	\end{proof}
}

\medskip

\begin{claim}\label{claim:conn_rvis}
	If $M$ is connected, then any edge $e_0=\{h_0,v_0\}\in E(G)$ is $r$-visible from some edge of $M$.
\end{claim}
\begin{proof}
	As $N_G(M_V\cup M_H)=V(G)$, there exists two vertices, $v_1\in M_V$ and $h_1\in M_H$, such that $\{v_1,h_0\},\{v_0,h_1\}\in E(G)$.

	\medskip

	If $v_0\in M_V$ or $h_0\in M_H$, then $\{v_0,h_1\}$ or $\{v_1,h_0\}$ is in $E(M)$.

	\medskip

	Otherwise, there exists a path in $M$, whose endpoints are $v_1$ and $h_1$, and this path and the edges $\{v_1,h_0\}$,$\{h_0,v_0\}$,$\{v_0,h_1\}$ form a cycle in $G$. By the bichordality of $G$, there exists a $C_4$ in $G$ which contains an edge of $M$ and $e_0$.
\end{proof}

\medskip

{\color{highlightnew}\Fref{claim:conn_rvis} implies that for connected $M$, we can select a subset of edges of $M$ that guard every edge of $G$. \Fref{case:disconnected} shows that if \Fref{thm:mainprime} holds for connected $M$, then it also holds when $M$ has multiple connected components. Furthermore, if \Fref{thm:mainprime} holds when $M$ is 2-connected, \Fref{case:conn} shows that theorem also holds when $M$ is connected.}

\medskip

{\color{highlightnew} Based on the level connectivity of $M$, we distinguish three cases: $M$ is 2-connected, $M$ is connected, and $M$ has multiple connected components. These cases and their proofs are quite different. When $M$ is connected or has multiple connected components, the proofs are relatively short and simple, and more importantly, only rely on elementary graph theory.

	\medskip

	The spirit of the proof dwells in \Fref{case:2conn}, which holds the deepest insight into the problem and is vastly longer and more complex than the other two cases following it. A few geometric arguments are present, but the overwhelming majority of reasoning in the 2-connected case is graph theoretic. Although this means that the proof is somewhat technical, we believe it is also quite robust, being built on the abstraction provided by $R$-trees and the pixelation graph.}

\subsection{\texorpdfstring{\boldmath $M$}{𝑀} is 2-connected}\label{case:2conn}
The $\frac43$ constant in the statement of \Fref{thm:mainprime} is determined by this case. Let us first present an outline of this case.

\medskip

{\color{highlightnew}
	First, we describe two fundamental properties of $M$ in \Fref{claim:D}~and~\ref{claim:DM}. Then, some of the horizontal slices of $M$ are refined into two thinner slices each, so as to avoid technical difficulties later in the proof. From then on, we work in this refined structure, denoted by $M'$ and $G'$. \Fref{claim:tau} provides the link between point guards of $G'$ and $G$. Next, we establish a relation  between edges of $M'$ (\Fref{def:dominance}), which describes when an edge can be replaced by another one, so that the replacement edge $r$-covers any edge seen exclusively by the replaced edge. This leads to the definition of ``hyperguards'' of $M'$ (\Fref{def:hyperguard}), which are proven to be point guard systems of $G'$ in \Fref{lemma:zprime}.

	\medskip

	After the lengthy preparation, we are finally ready to construct a hyperguard of $M'$ in \Fref{sec:constructing}. In the following \Fref{sec:estimating}, the size of the constructed hyperguard is estimated, finishing the proof of this case.
}

\medskip

If $E(M)$ consists of a single edge $e$, then $Z=\{e\}$ is clearly a point guard system of $G$ by~\Fref{claim:conn_rvis}.

\medskip

Suppose now, that $M$ has more than two vertices. Any edge of $M$ is contained in a cycle of $M$, and by the bichordality property, there is such a cycle of length 4. It is easy to see that the convex hull of the pixels determined by the edges of a $C_4$ is a rectangle. Define
\[ D_M=\bigcup\limits_{\{e_1,e_2,e_3,e_4\}\text{ is a $C_4$ in }M}\mathrm{Conv}\left(\bigcup\limits_{i=1}^4 \cap e_i\right).\]
The simply connectedness of $D$ implies that $D_M\subseteq D$.

\begin{claim}\label{claim:D}
	For any slice $s\in V(M)$ the intersection of $s$ and $D_M$ is connected.
\end{claim}
\begin{proof}
	Suppose that $e_1,e_2\in E(M)$ are such that $\cap e_1$ and $\cap e_2$ are in two different components of $s\cap D_M$. Since $M$ is 2-connected, there is a path connecting $e_1\setminus \{s\}$ and $e_2\setminus\{s\}$ in $M-s$.

	\medskip

	Take the shortest cycle in $M$ containing $e_1$ and $e_2$.
	If this cycle contains 4 edges, then the convex hull of their pixels is in $D_M$, which is a contradiction.
	Similarly, if the cycle contains more than 4 edges, the bichordality of $M$ implies that $s$ is joined to every second node of the cycle, which contradicts our assumption that $s\cap D_M$ is disconnected.
\end{proof}

{\color{highlightnew}
\begin{claim}\label{claim:DM}
	$D_M$ is simply connected.
\end{claim}
\begin{proof}
	Connectedness of $D_M$ follows from the connectedness of $M$ and \Fref{claim:D}. Suppose there is a hole in $D_M$. If the hole is a rectangle, the four slices of $M$ bounding it induce a $C_4$, which contradicts the definition of $D_M$.

	\medskip

	If the hole has more than 6 vertices, take a reflex vertex $x$ of it, and let $e\in E(M)$ be such that $x$ is a vertex of $\cap e$. Since $D_M\subseteq D$ and $D$ is simply connected, the horizontal slice of $e$ crosses the hole, and intersects another vertical slice of $M$. This contradicts \Fref{claim:D}.
\end{proof}}

\subsubsection{Splitting some slices of \texorpdfstring{$M_H$ and $M_V$}{𝑀\_𝐻 and 𝑀\_𝑉}}

Let $B_H\subset M_H$ be the set of those slices whose top and bottom sides both intersect $\partial D_M$ in an uncountable number of points of $\mathbb{R}^2$. Similarly, let $B_V\subset M_V$ be the set of those slices whose left and right sides both intersect $\partial D_M$ in an uncountable number of points.

\medskip

For technical reasons, we split each element of $h\in B_H$ horizontally through $c(h)$ to get two isometric rectangles in $\mathbb{R}^2$; let the set of the resulting refined horizontal slices be $B'_H$. Similarly, we get $B'_V$ by splitting elements of $B_V$ vertically through their centroids. Also, we define
\begin{align*}
	A'_H&=B'_H\bigcup A_H\setminus B_H,\\
	A'_V&=B'_V\bigcup A_V\setminus B_V,\\
	M'_H&=B'_H\bigcup M_H\setminus B_H,\\
	M'_V&=B'_V\bigcup M_V\setminus B_V.
\end{align*}
\textcolor{highlightnew}{Let the $R$-tree on $A'_H$ and $A'_V$ be $T'_H$ and $T'_V$, respectively.} Let $\tau$ map $h\in B'_H$ to $\tau(h)\in A_H$ for which $h\subseteq \tau(h)$ holds, and let $\tau$ be the identity function on $A'_H\setminus B'_H$. Define $\tau$ analogously on $A'_V$.

\medskip

Let $G'$ be the intersection graph of $A'_H$ and $A'_V$ (as in the statement of \Fref{thm:mainprime}). Also, let $M'=G'[M'_H\cup M'_V]=\tau^{-1}(M)$. Observe that $\tau$ naturally defines a graph homomorphism $\tau:G'\to G$ (edges are mapped vertex-wise).

\medskip

\begin{claim}\label{claim:tau}
	In $G'$, the set $M'_H$ dominates $A'_V$, and $M'_V$ dominates $A'_H$.
	Furthermore, if $Z'\subseteq E(M')$ is a point guard system of $G'$, then $Z=\tau(Z')\subseteq E(M)$ is a point guard system of $G$.
\end{claim}
\begin{proof}
	The first statement of this claim holds, since $\tau$ maps non-edges to non-edges, and both $M'_H=\tau^{-1}(M_H)$ and $M'_V=\tau^{-1}(M_V)$ by definition.
	As $\tau$ is a graph homomorphism, it preserves $r$-visibility, which implies the second statement of this claim.
\end{proof}

\medskip

Notice, that $M'$ is 2-connected and $D_M=D_{M'}$.
It is straightforward to verify that an edge $e\in E(M')$ falls into one of the following 4 categories:
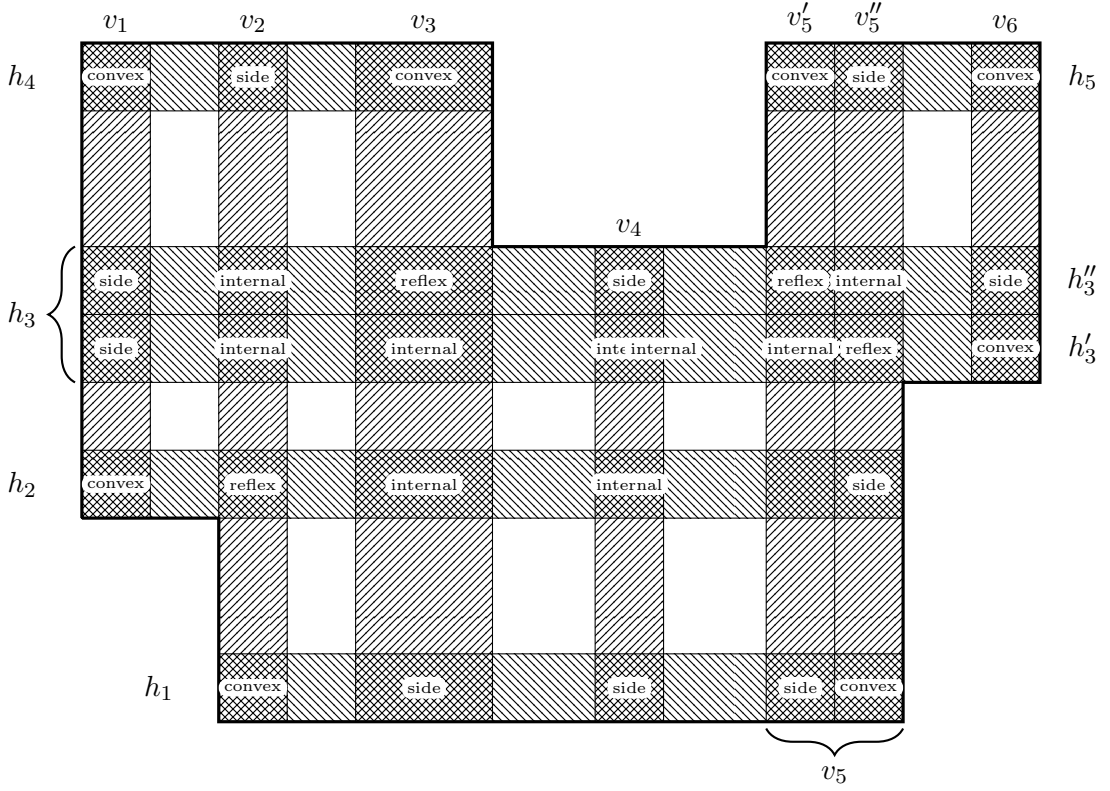
\begin{figure}[H]
	\centering
	\begin{tikzpicture}[scale=0.9]
		\begin{scope}
			\draw[very thick] (1,3) -- (1,10) -- (7,10) -- (7,7) -- (11,7) -- (11,10) -- (15,10) -- (15,5) -- (13,5) -- (13,0) -- (3,0) -- (3,3) -- (1,3);

			\def\hpat{north east lines}
			\def\vpat{north west lines}
			\filldraw[thin,pattern=\hpat] (1,3) rectangle (2,10);
			\filldraw[thin,pattern=\hpat] (3,0) rectangle (4,10);
			\filldraw[thin,pattern=\hpat] (5,0) rectangle (7,10);
			\filldraw[thin,pattern=\hpat] (8.5,0) rectangle (9.5,7);
			\filldraw[thin,pattern=\hpat] (11,0) rectangle (13,10);
			\draw[thin] (12,0) -- (12,10);
			\filldraw[thin,pattern=\hpat] (14,5) rectangle (15,10);

			\filldraw[thin,pattern=\vpat] (3,0) rectangle (13,1);
			\filldraw[thin,pattern=\vpat] (1,3) rectangle (13,4);
			\filldraw[thin,pattern=\vpat] (1,5) rectangle (15,7);
			\draw[thin] (1,6) -- (15,6);
			\filldraw[thin,pattern=\vpat] (1,9) rectangle (7,10);
			\filldraw[thin,pattern=\vpat] (11,9) rectangle (15,10);

			\draw [decorate,thick,decoration={brace,amplitude=10pt}]
(13,-0.1) -- (11,-0.1);

			\foreach \x/\y/\l in {2/10/1,4/10/2,6.5/10/3,9.5/7/4,15/10/6}
			\draw (\x-0.5,\y) node[anchor=south] {$v_\l$};
			
			\draw (12,-0.5) node[anchor=north] {$v_5$};
			\draw (12-0.5,10) node[anchor=south] {$v'_5$};
			\draw (13-0.5,10) node[anchor=south] {$v''_5$};

			\draw [decorate,thick,decoration={brace,amplitude=10pt}]
			(1-0.1,5) -- (1-0.1,7);

			\foreach \x/\y/\l in {3/0/1,1/3/2,1/5.5/3,1/9/4,16.5/9/5}
			\draw (\x-0.5,\y+0.5) node[anchor=east] {$h_\l$};

			\draw (16,5+0.5) node[anchor=east] {$h'_3$};
			\draw (16,6+0.5) node[anchor=east] {$h''_3$};

			\foreach \x/\y in {2/4,2/10,6.5/10,4/1,13/1,12/10,15/10,15/6}
			\draw (\x-0.5,\y-0.5) node[rectangle,rounded corners=3pt,fill=white,text opacity=1,fill opacity=1,inner sep=2pt] {\tiny convex};

			\foreach \x/\y in {2/6,2/7,15/7,4/10,13/4,9.5/7,9.5/1,6.5/1,12/1,13/10}
			\draw (\x-0.5,\y-0.5) node[rectangle,rounded corners=3pt,fill=white,text opacity=1,fill opacity=1,inner sep=2pt] {\tiny side};

			\foreach \x/\y in {4/4,6.5/7,12/7,13/6}
			\draw (\x-0.5,\y-0.5) node[rectangle,rounded corners=3pt,fill=white,text opacity=1,fill opacity=1,inner sep=2pt] {\tiny reflex};

			\foreach \x/\y in {4/6,4/7,6.5/4,6.5/6,9.5/4,9.5/6,13/7,12/6,10/6}
			\draw (\x-0.5,\y-0.5) node[rectangle,rounded corners=3pt,fill=white,text opacity=1,fill 	opacity=1,inner sep=2pt] {\tiny internal};

		\end{scope}
	\end{tikzpicture}
	\caption{We have $M_H=\{h_1,\dots,h_5\}$, $M'_H=M_H-h_3+h_3'+h_3''$, $M_V=\{v_1,\dots,v_6\}$, $M'_V=M_V-v_5+v'_5+v''_5$. The thick line is the boundary of $D_M$. Each rectangle pixel is labeled according to the type of its corresponding edge of $M'$.}\label{fig:types}
\end{figure}

\begin{description}\setlength{\baselineskip}{1em}
	\item[Convex edge:] 3 vertices of $\cap e$ fall on $\partial D_M$, e.g., the edge $\{h_2,v_1\}$ on \Fref{fig:types};

	\item[Reflex edge:] exactly 1 vertex of $\cap e$ falls on $\partial D_M$, e.g., $\{h_3'',v_3\}$ on \Fref{fig:types};

	\item[Side edge:] two neighboring vertices of $\cap e$ fall on $\partial D_M$,  e.g., $\{h_1,v_4\}$ on \Fref{fig:types};

	\item[Internal edge:] zero vertices of $\cap e$ fall on $D_M$,  e.g., $\{h_2,v_3\}$ on \Fref{fig:types}.
\end{description}

Notice that on \Fref{fig:types}, the edge $\{h_3,v_5\}$ falls into neither of the previous categories, as two non-neighboring (diagonally opposite) vertices of pixel $h_3\cap v_5$ fall on $D_M$. This clearly cannot happen with edges of $G'$, but $G$ may contain edges of this type. The $\tau$ preimage of such an edge is a set of two reflex and two internal edges of $M'$.

\medskip

The preimages of a convex edge are a convex edge and a side edge ($M'$ is 2-connected), the preimages of a side edge are two side edges, and the preimages of a reflex edge are a reflex edge and an internal edge. In the other direction, $\tau$ maps convex edges to convex edges, and side edges to convex or side edges.

\medskip

The following definition allow us to break our proof into smaller, transparent parts, which ultimately boils down to presenting a precise proof. It captures a condition which in certain circumstances allows us to conclude that a guard $e_1$ can be replaced by $e_2$ such that we still have complete coverage of $G'$.

\begin{definition}\label{def:dominance} We say that a slice $s_0$ is \textbf{between} slices $s_1$ and $s_2$ (all vertical or horizontal), if in the corresponding $R$-tree $s_0$ is on the path between $s_1$ and $s_2$. For any two edges $e_1,e_2\in E(M')$, where $e_1=\{v_1,h_1\}$ and $e_2=\{v_2,h_2\}$, we write $e_2 \covers e_1$ ($e_2$ dominates $e_1$) iff either
	\begin{itemize}
		\item $h_1=h_2$, and $\exists h_3,h_4\in M'_H$ such that $\{v_1,v_2,h_3,h_4\}$ induces a $C_4$ in $M'$, and $h_1=h_2$ is between $h_3$ and $h_4$; or

		\item $v_1=v_2$, and $\exists v_3,v_4\in M'_V$ such that $\{v_3,v_4,h_1,h_2\}$ induces a $C_4$ in $M'$, and $v_1=v_2$ is between $v_3$ and $v_4$; or

		\item $e_1\cap e_2=\emptyset$, and $\exists v_3\in M'_V$ and $h_3\in M'_H$ such that both $\{v_1,h_2,v_2,h_3\}$ and $\{h_1,v_3,h_2,v_2\}$ induces a $C_4$ in $M'$;
		      furthermore, $v_1$ is between $v_2$ and $v_3$, and $h_1$ is between $h_2$ and $h_3$.
	\end{itemize}
	We write $e_2 \rela e_1$ iff both $e_2 \covers e_1$ and $e_1 \covers e_2$ hold. Note that $\rela$ is a symmetric, but generally intransitive relation. \textcolor{highlightnew}{For convenience, we define both relations to be reflexive.}
\end{definition}

For example, on \Fref{fig:types}, $\{h_1,v_3\}\rela \{h_3'',v_3\}$, and $\{h_1,v_2\}\rightarrow \{h_3'',v_3\}$.
Also, $\{h_3'',v_3\}\rela \{h_3'',v_1\}$, but $\{h_3'',v_3\}\not\rightarrow \{h_3',v_1\}$. This is a technicality which makes the proofs easier, but does not cause any issues in the end, as $\tau(\{h_3'',v_1\})=\tau(\{h_3',v_1\})$. {\color{highlightnew} The fact that $\{h_3'',v_3\}\rightarrow \{h_3',v_5''\}$ and $\{h_3',v_5''\}\not\rightarrow \{h_3'',v_3\}$ shows that $\rightarrow$ is not symmetric.}

\subsubsection{Hyperguards}\label{sec:hyperguards}

We will search for a point guard system of $M'$ with very specific properties, which are described by the following definition.

\begin{definition}\label{def:hyperguard}
	Suppose $Z'\subseteq E(M')$ is such, that
	\begin{enumerate}
		\item $Z'$ contains every convex edge of $M'$,
		\item for any non-internal edge $e_1\in E(M')\setminus Z'$, there exists some $e_2\in Z'$ for which $e_2\covers e_1$, and
		\item\label{item:internal} \textcolor{highlightnew}{if $h_3,h_4\in M'_H$ are neighboring slices in the $R$-tree on $M'_H$, and $v_3,v_4$ are the end-nodes of the path induced by $N_{M'}(h_3)\bigcap N_{M'}(h_4)$ in the $R$-tree on $M'_V$, and $\{v_3,h_3,v_4,h_4\}$ induces a $C_4$ in $M'$, then there exists $e_2\in Z'$ such that $e_2\rightarrow \{v_3,h_3\},\{v_4,h_3\}$ or $e_2\rightarrow \{v_3,h_4\},\{v_4,h_4\}$ holds.}
	\end{enumerate}
	If these three properties hold, we call $Z'$ a \textbf{hyperguard} of $M'$.
\end{definition}

The \ref{item:internal}$^\text{rd}$ property of hyperguards corresponds to the  configuration in $D_M$ shown on \Fref{fig:neck}.  Observe that symmetry between horizontal and vertical slices is broken this property.
\begin{figure}[H]
	\centering
	\begin{tikzpicture}[scale=1]
		\begin{scope}
			\draw[very thick] (-1,0) -- (0,0) -- (0,-2) -- (-1,-2);
			\draw[very thick] (6,0) -- (5,0) -- (5,-2) -- (6,-2);

			\def\hpat{north east lines}
			\def\vpat{north west lines}
			\fill[pattern=\hpat] (-2,0) rectangle (7,1);
			\fill[pattern=\hpat] (-2,-3) rectangle (7,-2);
			\fill[pattern=\vpat] (0,2) rectangle (1,-4);
			\fill[pattern=\vpat] (2,2) rectangle (3,-4);
			\fill[pattern=\vpat] (4,2) rectangle (5,-4);

			\foreach \x/\y in {0/0,4/0,0/-3,4/-3}
			\draw (\x+0.5,\y+0.5) node[rectangle,rounded corners=3pt,fill=white,text opacity=1,fill opacity=1,inner sep=2pt] {\tiny reflex};

			\foreach \x/\y in {2/0,2/-3}
			\draw (\x+0.5,\y+0.5) node[rectangle,rounded corners=3pt,fill=white,text opacity=1,fill 	opacity=1,inner sep=2pt] {\tiny internal};

			\foreach \x/\y/\z in {-2/0/$h_3$,-2/-3/$h_4$,0/-5/$v_3$,2/-5/$v_5$,4/-5/$v_4$}
			\draw (\x+0.5,\y+0.5) node[rectangle,rounded corners=3pt,fill=white,text opacity=1,fill 	opacity=1,inner sep=2pt] {\z};

			%			\draw[thin] (-2,1) -- (7,1);
			%			\draw[thin] (-2,-3) -- (7,-3);
			%			\foreach \x in {0,1,2,3,4,5}
			%				\draw[thin] (\x,2) -- (\x,-4);

		\end{scope}
	\end{tikzpicture}
	\caption{\color{highlightnew} A neck in $D_M$. There are no horizontal slices of $M'_H$ between $h_3$ and $h_4$, but there can be vertical slices between $v_3$ and $v_4$.}\label{fig:neck}
\end{figure}
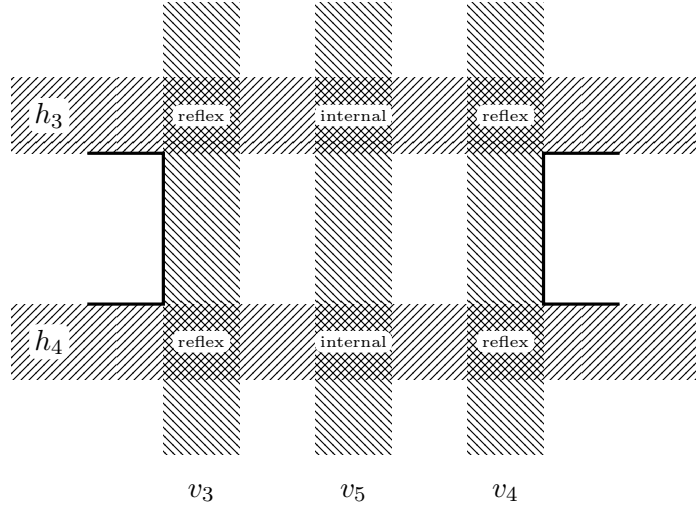

\begin{lemma}\label{lemma:zprime}
	Any hyperguard $Z'$ of $M'$ is a point guard system of $G'$, i.e., any edge of $G'$ is $r$-visible from some element of $Z'$.
\end{lemma}
\begin{proof}
	Let $e_0=\{v_0,h_0\}\in E(G')$ be an arbitrary edge. By~\Fref{claim:conn_rvis}, there exists an edge
	$e_1=\{v_1,h_1\}\in E(M')$ which has $r$-vision of $e_0$, and we also \textcolor{highlightnew}{suppose that $e_1$ is chosen so that $\mathrm{dist}_{T'_H}(h_0,h_1)+\mathrm{dist}_{T'_V}(v_0,v_1)$ is minimal.}

	\medskip

	Trivially, if $e_1\in Z'$ (for example, if $e_1$ is a \textbf{convex edge} of $M'$), then $e_0$ is $r$-visible from $e_1$. Assume now, that $e_1\notin Z'$.
	\begin{itemize}
		\item If $e_1$ is a \textbf{reflex or side edge} of $M'$, then $\exists e_2=\{v_2,h_2\}\in Z'$ so that $e_2\covers e_1$. We claim that $e_2$ has $r$-vision of $e_0$ in $G'$ (this is the main motivation for \Fref{def:dominance}).
		      \begin{enumerate}
			      \item If $h_1=h_2$: by the choice of $e_1$ and $e_2$, $v_1$ is joined to $h_0,h_3,h_4$ in $G'$. Since $N_{G'}(v_1)$ is the vertex set of a path in $T'_H$, the choice of $e_1$ guarantees that $h_0$ is between $h_3$ and $h_4$, which are neighbors of $v_2$ in $G'$. Therefore $\{ v_2, h_0\}\in E(G')$, so $\{v_0,h_0,v_2,h_1(=h_2)\}$ induces a $C_4$ in $G'$.

			      \item If $v_1=v_2$: the proof proceeds analogously to the previous case.

			      \item If $e_1\cap e_2=\emptyset$: by the choice of $e_1$ and $e_2$, $v_1$ is joined to $h_0,h_3,h_2$ in $G'$, and $v_1$ is joined to $v_0,v_3,v_2$ in $G'$. The choice of $e_1$ guarantees that $h_0$ is between $h_3$ and $h_2$, and that $v_0$ is between $v_3$ and $v_2$. Therefore $\{ v_2\cap h_0\},\{v_0\cap h_2\}\in E(G')$, so $\{v_0,h_0,v_2,h_2\}$ induces a $C_4$ in $G'$.
		      \end{enumerate}
		      In any of the three cases, $e_0$ is $r$-visible from $e_2$ in $G'$.

		\item If $e_1$ is an \textbf{internal edge} of $M'$, then $\cap e_0\subset D_M$, so $\cap e_0$ is in a rectangle corresponding to a $C_4$ of $M'$. \color{highlightnew} Thus there are two elements $h_3,h_4\in M'_H\cap N_{G'}(v_0)$ such that there does not exist an element of $M'_H$ which is between $h_3$ and $h_4$, but $h_0$ is between $h_3$ and $h_4$ (or is equal to one of them). Because $M'$ is 2-connected, \Fref{claim:treeprop} applies. Let the end-points of the path induced by $N_{M'}(h_3)\bigcap N_{M'}(h_4)$ be $v_3$ and $v_4$.

		      \medskip

		      We claim that the edges of the $C_4$ induced by $\{v_3,h_3,v_4,h_4\}$ are non-internal edges. Take $\{v_3,h_3\}$, for example.
		      \begin{itemize}
			      \item If $v_3$ is an end-point of the path induced by $N_{M'}(h_3)$, then \Fref{claim:D} implies that one of the sides of the pixel $v_3\cap h_3$ is a subset of $\partial D_M$. In other words, $\{v_3,h_3\}$ is a side or a convex edge of $M'$.
			      \item Otherwise, there is a neighbor $v_5$ of $v_3$ in the path induced by $N_{M'}(h_3)$ in the $R$-tree on $M'_H$, such that $v_5\notin N_{M'}(h_4)$. If $\{v_3,h_3\}$ is an internal edge, then $\{v_5,h_4\}\in E(M')$, so $\{v_3,h_3\}$ can only be a reflex edge.
		      \end{itemize}
		      The same reasoning holds for the other three edges induced by $\{v_3,h_3,v_4,h_4\}$. Clearly, $e_0$ is $r$-visible to all four edges; if any of them is a convex edge, we are done.

		      \medskip

		      If, say, $\{v_3,h_3\}$ is a side edge, then $\exists e_2=\{v_2,h_2\}\in Z'$ such that $e_2\rightarrow \{v_3,h_3\}$. Because $v_3$ is an end-point of the path induced by $N_{M'}(h_3)$ in the $R$-tree on $M'_H$, we must have $h_2=h_3$. There are two horizontal slices $h_5,h_6\in M'_H$ which intersect both $v_2$ and $v_3$, and $h_3$ is between them. Both $N_{M'}(v_2)$ and $N_{M'}(v_3)$ are the vertex set of a path in the $R$-tree on $M'_H$, and so is their intersection $N_{M'}(v_2)\cap N_{M'}(v_3)$. It contains the vertices of the path from $h_5$ to $h_6$ through $h_3$, therefore it contains $h_4$ (there is no slice of $M'_H$ between $h_3$ and $h_4$). Thus $e_2$ has $r$-vision of the four induced edges of $\{v_3,h_3,v_4,h_4\}$, and consequently, of $e_0$.

		      \medskip

		      If each of the four induced edges of $\{v_3,h_3,v_4,h_4\}$ are reflex edges, then without loss of generality, we may assume that there $\exists e_2=\{v_2,h_2\}\in Z'$ such that $e_2\rightarrow \{v_3,h_3\},\{v_4,h_3\}$. This implies that $v_2,v_3,v_4\in N_{G'}(h_3)$. If $v_2$ is between $v_3$ and $v_4$ (or is equal to one of them), then $v_2\in N_{G'}(h_4)$, so $e_2$ has $r$-vision of each of the four induced edges of $\{v_3,h_3,v_4,h_4\}$ and of $e_0$.

		      \medskip

		      Suppose now, that $v_2$ is not between $v_3$ and $v_4$, i.e., $v_2\notin N_{M'}(h_3)\cap N_{M'}(h_4)$. Thus $h_2$ is not equal to either $h_3$ or $h_4$, and so cannot be between them. Because $e_2\rightarrow \{v_3,h_3\}$, there is an $h_5$ such that $\{v_3,h_5\}\in E(M')$, and $h_3$ is between $h_2$ and $h_5$ (all of which are joined to $v_3$ in $M'$). By construction, $h_4$ is between $h_2$ and $h_5$. Since $v_2$ is joined to both $h_2$ and $h_5$, it should be joined to $h_4$, a contradiction.
	\end{itemize}
	We have verified the statement in every case, so the proof of this lemma is complete.
\end{proof}

\textcolor{highlightnew}{Observe that if $D_M$ does not contain a ``neck'' (see \Fref{fig:neck}), even the first two properties of a hyperguard are sufficient to prove \Fref{lemma:zprime}.}

\medskip

Notice, that the set of all convex, reflex, and side edges of $E(M')$ form a hyperguard of $M'$. By \Fref{lemma:zprime}, this set is a point guard system of $G'$, and \Fref{claim:tau} implies that its $\tau$-image is a point guard system of $G$. The cardinality of the $\tau$-image of this hyperguard is bounded by $2|V(M)|-4$ (we will see this shortly), which is already a magnitude lower than what the trivial choice of $E(M)$ would give (generally, $|E(M)|$ can be equal to $\Omega(|V(M)|^2)$).

\medskip

Let the number of convex, side, and reflex edges in $M'$ be $c'$, $s'$, and $r'$, respectively. \Fref{claim:D} and \Fref{claim:DM} allow us to count these objects.

\begin{enumerate}
	\item The number of reflex vertices of $D_M$ is equal to $r'$: any reflex vertex is a vertex of a reflex edge, and the way $M'$ and $D_M$ is constructed guarantees that exactly one vertex of the pixel of a reflex edge is a reflex vertex of $D_M$.

	\item The number of convex vertices of $D_M$ is equal to $c'$: any convex vertex is a vertex of the pixel of a convex edge, and the way $D_M$ is constructed guarantees that exactly one vertex of the pixel of a convex edge is a convex vertex.

	\item The cardinality of $V(M')$ is $c'+\frac12s'$: the first and last edge incident to any element of $V(M')$ ordered from left-to-right (for elements of $M'_H$) or from top-to-bottom (for elements of $M'_V$) is a convex or a side edge. Conversely, any convex edge is the first or last incident edge of exactly one element of $M'_H$ and one element of $M'_V$. A side edge is the first or last incident edge of exactly one element of $V(M')$.

	\item For any reflex edge $e_1=\{v_1,h_1\}\in E(M')$, there is exactly one reflex or side edge in $E(M')$ which contains $v_1$ and is in the $\rela$ relation with $e_1$, and the same can be said about $h_1$.

	\item Any side edge $e_1\in E(M')$ is in $\rela$ relation with exactly one reflex or side edge which it intersects. The intersection is the slice in $V(M')$ on which $e_1$ is a boundary edge.
\end{enumerate}

%Observe, that except the first statement of the previous list, the statements hold even for $\tau$ images.

We can now compute the size of the set of all convex, reflex, and side edges of $M'$:
\[ c'+r'+s'=2c'-4+s'=2|V(M')|-4.\]
Furthermore, it is clear that taking the $\tau$-image of this set decreases its cardinality by $2|B_H|+2|B_V|$ (new reflex and side edges are created at both ends of slices in $B_H$ and $B_V$ when splitting them). Thus the cardinality of the $\tau$-image of all convex, reflex, and side edges of $M'$ is at most
\[ 2|V(M')|-4-2|B_H|-2|B_V|= 2|V(M)|-4, \] proving the claim from the previous page. Readers who are only interested in a result which is sharp up to a constant factor, may skip to \Fref{case:conn}. Further analysis of $M'$ allows us to lower the coefficient $2$ to $\frac43$.

\medskip

Define the {\bfseries auxiliary graph \boldmath $X$} as follows: let $V(X)$ be the set of reflex and side edges of $M'$, and let
\[ E(X)=\Big\{\{e,f\}:\ e\neq f,\ e\cap f\neq\emptyset,\ e\rela f\Big\}. \]
By our observations, $X$ is the disjoint union of some cycles and $\frac12s'$ paths. This structure allows us to select a hyperguard which contains a subset of the reflex and side edges of $M'$, instead of the whole set.

\medskip

{\color{highlightnew}
	In the next section, we use the following trivial fact several times.
	\begin{claim}\label{claim:pathdom}
		A path on $k$ nodes has a dominating set of size \[\left\lceil\frac{k}{3}\right\rceil=\left\lfloor \frac{k+2}{3}\right\rfloor.\]
	\end{claim}
}

\subsubsection{Constructing a hyperguard \texorpdfstring{\boldmath $Z'$ of $M'$}{𝑍' of 𝑀'}.}\label{sec:constructing}  We will define  ${(Z'_j)}_{j=0}^\infty$, a sequence of (set theoretically) increasing sequence of subsets of $E(M')$, and ${(X_j)}_{j=0}^\infty$, a decreasing sequence of induced subgraphs of $X$.

\medskip

Additionally, we will define a function $w_j:V(X)\to \{0,1,2\}$, and extend its domain to any subgraph $H\subseteq X$ by defining $w_j(H)=\sum_{e\in V(H)}w_j(e)$. The purpose of $w_j$, very vaguely, is that as $Z'$ will contain every third node of $X$, we need to keep count of the modulo 3 remainders. Furthermore, $w_j$ serves as buffer in a(n implicitly defined) weight function (see \fref{ineq:recursion}).

\medskip

For a set $E_0\subseteq E(X)$, let the indicator function of $E_0$ be
\[
	\mathds{1}_{E_0}(e)=\left\{
	\begin{array}{ll}
		1,\quad & \text{if }e\in E_0,               \\
		0,\quad & \text{if }e\in E(X)\setminus E_0.
	\end{array}
	\right.
\]

\medskip

Let $Z'_0=\emptyset$ and $X_0=X$. By our previous observations, $X$ does not contain isolated nodes. Define $w_0:V(X)\to \{0,1,2\}$ such that
\[
	w_0(e)=\left\{
	\begin{array}{ll}
		1,\quad & \text{if }d_{X_0}(e)=1,             \\
		0,\quad & \text{if }d_{X_0}(e)=0\text{ or }2.
	\end{array}
	\right.
\]

\medskip

In the $j^\text{th}$ step, we will define $Z'_j$, $X_j$, and $w_j$ so that
\begin{itemize}
	\item $Z'_{j-1}\subseteq Z'_{j}$, $X_{j}\subseteq X_{j-1}$,
	\item $\{e\in V(X_j)\ |\ d_{X_j}(e)=1\}\subseteq w^{-1}_j(1)$,
	\item $\{e\in V(X_j)\ |\ d_{X_j}(e)=0\}=w^{-1}_j(2)$, and
	\item $\forall e_0\in V(X)\setminus V(X_j)$, either $e_0\in Z'_j$, or $\exists e_1\in Z'_j$ so that $e_1\covers e_0$.
\end{itemize}
If these hold, then for any path component $P_j$ in $X_j$, we have $w_j(P_j)\ge 2$.

\medskip

{\color{highlightnew} As $j$ increases, the construction goes through 5 phases. In each of Phases 2-4, $j$ is incremented for multiple iterations, until $X_j$ satisfies some predefined condition. The different phases and the relevant parts of $D_M$ are depicted on \Fref{fig:phases}.}

\begin{figure}[H]
	\centering

	\begin{subfigure}{\textwidth}
		\centering
		\begin{tikzpicture}[scale=0.75]
			\begin{scope}
				\draw[very thick] (0,0) -- (0,4);
				\draw[very thick] (10,0) -- (10,1) -- (12,1) -- (12,4);
				\draw[very thick] (3,0) -- (3,1) -- (4,1) -- (4,0);
				\draw[very thick] (7,4) -- (7,3) -- (8,3) -- (8,4);

				\def\hpat{north east lines}
				\def\vpat{north west lines}

				\fill[pattern=\hpat] (0,1) rectangle (12,3);

				\draw[thin] (0,2) -- ++(12,0);

				\fill[pattern=\vpat] (0,0) rectangle ++(1,4);
				\fill[pattern=\vpat] (2,0) rectangle ++(1,4);
				\fill[pattern=\vpat] (4,0) rectangle ++(1,4);
				\fill[pattern=\vpat] (6,0) rectangle ++(1,4);
				\fill[pattern=\vpat] (8,0) rectangle ++(2,4);
				\fill[pattern=\vpat] (11,1) rectangle ++(1,3);
				
				\draw[thin] (9,0) -- ++(0,4);

				\draw [decorate,thick,decoration={brace,amplitude=10pt}]
				(-0.1,1) -- (-0.1,3);

				\draw (-0.5,2) node[anchor=east] {$B_H\ni h=h'\cup h''$};
				\draw (12,2.5) node[anchor=west] {$h'\in B_H'$};
				\draw (12,1.5) node[anchor=west] {$h''\in B_H'$};

				\foreach \x/\y/\z in {0/2/$e_1$,0/1/$e_2$,2/1/$e_3$,4/1/$e_4$,6/2/$e_5$,8/2/$e_6$,9/1/$e_7$,11/2/$e_8$,11/1/$e_9$}
				\draw (\x+0.5,\y+0.5) node[rectangle,rounded corners=3pt,fill=white,text opacity=1,fill 	opacity=1,inner sep=2pt] {\z};

			\end{scope}
		\end{tikzpicture}
		\caption{\Fref{phase:initial}: handling the new reflex and side edges created on the refined slices. We have $e_1,e_2\in S'$, $e_9\in C'$, $e_8\in U'$, and $e_3\in Q'$.}\label{fig:phase1}
	\end{subfigure}

	\bigskip

	\begin{subfigure}{0.5\textwidth}
		\flushleft
		\begin{tikzpicture}[scale=0.75]
			\begin{scope}
				\draw[very thick] (3,2) -- (3,1) -- (4,1);
				\draw[very thick] (-1,1) -- (0,1) -- (0,2);
				\draw[very thick] (-1,-2) -- (0,-2) -- (0,-3);

				\draw[very thick] (4,-5) -- (3,-5) -- (3,-6);
				\draw[very thick] (6,-3) -- (6,-2) -- (7,-2);

				\def\hpat{north east lines}
				\def\vpat{north west lines}
				\fill[pattern=\hpat] (-1,0) rectangle (4,1);
				\fill[pattern=\hpat] (-1,-2) rectangle (7,-1);
				\fill[pattern=\hpat] (1,-5) rectangle (4,-4);
				\fill[pattern=\vpat] (0,2) rectangle (1,-3);
				\fill[pattern=\vpat] (2,2) rectangle (3,-6);
				\fill[pattern=\vpat] (5,-3) rectangle (6,0);

				\foreach \x/\y/\z in {0/-2/$e_1$,0/0/$e_2$,2/0/$e_3$,2/-2/$f$,2/-5/$e_4$,5/-2/$e_{2k}$}
				\draw (\x+0.5,\y+0.5) node[rectangle,rounded corners=3pt,fill=white,text opacity=1,fill 	opacity=1,inner sep=2pt] {\z};

			\end{scope}
		\end{tikzpicture}
		\caption{\Fref{phase:cycle}: cutting cycles}\label{fig:phase2}
	\end{subfigure}%
	\begin{subfigure}{0.5\textwidth}
		\flushright
		\begin{tikzpicture}[scale=0.75]
			\begin{scope}

				\def\hpat{north east lines}
				\def\vpat{north west lines}

				\fill[pattern=\hpat] (-1,0) rectangle (4,1);
				\fill[pattern=\hpat] (-1,-2) rectangle (7,-1);
				\fill[pattern=\hpat] (1,-5) rectangle (4,-4);
				\fill[pattern=\vpat] (0,2) rectangle (1,-3);
				\fill[pattern=\vpat] (2,2) rectangle (3,-5);
				\fill[pattern=\vpat] (5,-3) rectangle (6,0);
				\fill[white] (2,-5) rectangle (3,-6);

				\foreach \x/\y/\z in {0/-2/$e_4$,0/0/$e_3$,2/0/$e_2$,2/-2/$f$,2/-5/$e_1$,5/-2/$e_{5}$}
				\draw (\x+0.5,\y+0.5) node[rectangle,rounded corners=3pt,fill=white,text opacity=1,fill 	opacity=1,inner sep=2pt] {\z};

				\draw[very thick] (3,2) -- (3,1) -- (4,1);
				\draw[very thick] (-1,1) -- (0,1) -- (0,2);
				\draw[very thick] (-1,-2) -- (0,-2) -- (0,-3);

				\draw[very thick] (1,-5) -- (4,-5);
				\draw[very thick] (6,-3) -- (6,-2) -- (7,-2);

			\end{scope}
		\end{tikzpicture}
		\caption{\Fref{phase:selfintersect}: cutting self-intersecting paths}\label{fig:phase3}
		% TODO: we need a figure for cutting off a 4-cycle, and 2 or 3 figures for when we find a chord
	\end{subfigure}

	\bigskip

	\begin{subfigure}{\textwidth}
		\centering
		\begin{tikzpicture}[scale=0.75]
			\begin{scope}
				\draw[very thick] (-1,0) -- (0,0) -- (0,1) -- (-2,1) -- (-2,3) -- (-1,3) -- (-1,5) -- (0,5) -- (0,6) -- (3,6) -- (3,3) -- (4,3) -- (4,6) -- (7,6) -- (7,5) -- (8,5) -- (8,3) -- (9,3) -- (9,1)  -- (7,1)  -- (7,0) -- (8,0);

				\def\hpat{north east lines}
				\def\vpat{north west lines}
				\fill[pattern=\hpat] (-2,1) rectangle (9,2);
				\fill[pattern=\hpat] (-1,-1) rectangle (8,0);
				\fill[pattern=\vpat] (0,6) rectangle (1,-2);
				\fill[pattern=\vpat] (6,6) rectangle (7,-2);

				\fill[pattern=\hpat] (-1,4) rectangle (3,5);
				\fill[pattern=\hpat] (4,4) rectangle (8,5);
				\fill[pattern=\vpat] (2,6) rectangle (3,3);
				\fill[pattern=\vpat] (4,6) rectangle (5,3); 

				\foreach \x/\y/\z in {-3/1/$h_3$,-2/-1/$h_4$,0/-3/$v_3$,6/-3/$v_4$,2/4/$e_1$,0/4/$e_2$,0/1/$e_3$,6/1/$e_4$,6/4/$e_5$,4/4/$e_6$}
				\draw (\x+0.5,\y+0.5) node[rectangle,rounded corners=3pt,fill=white,text opacity=1,fill 	opacity=1,inner sep=2pt] {\z};
			\end{scope}
		\end{tikzpicture}
		\caption{\Fref{phase:3rdproperty}: covering necks. Some slices are not shown or drawn completely to avoid clutter. The set $\{e_i\ :\ i=1,\dots,6\}$ induces a path in $X$, and $\{e_2,e_5\}$ is its minimum dominating set.}\label{fig:phase4}
	\end{subfigure}%
	%	\begin{subfigure}{0.5\textwidth}
	%		\caption{Phase 5: removing the remaining paths}\label{fig:phase5}
	%	\end{subfigure}

	\caption{Demonstrating possible substructures of $X$ which are handled in Phases 1-4.}\label{fig:phases}
\end{figure}
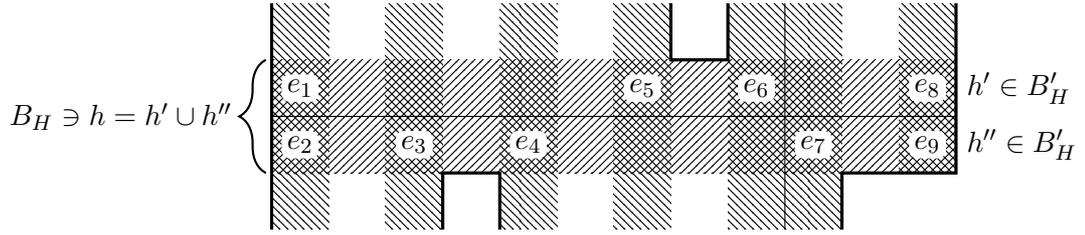
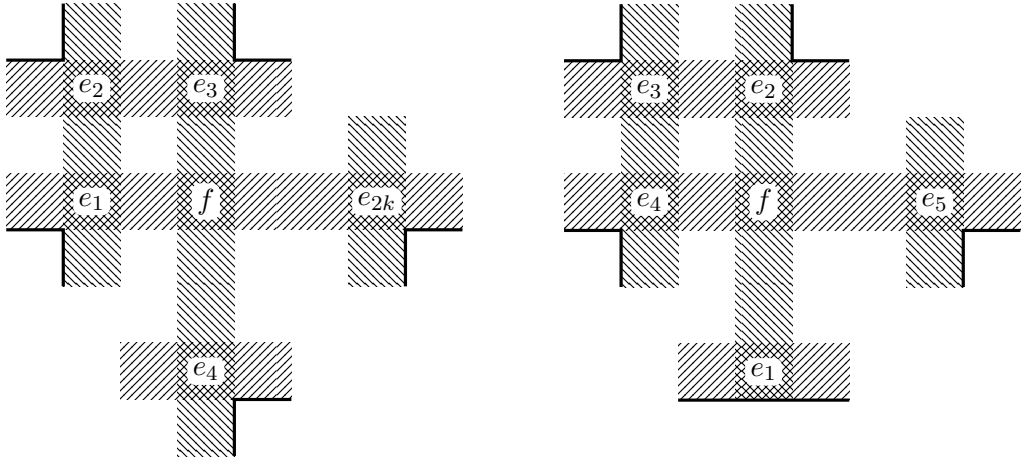
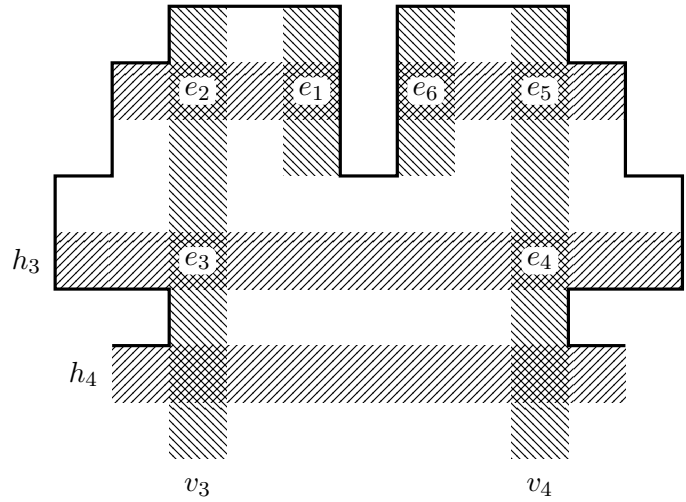

\phase{}\label{phase:initial} Let the set of convex edges of $M'$ be $C'$.  Let
\begin{align*}
	S' & =\Big\{ e\in V(X):\ \tau(e)\text{ is a side edge}\Big\},                                                                                \\
	U'&=V(X)\bigcap \tau^{-1}(\tau(C'))                                                                                                      \\
	W' & = \bigcup_{\substack{e_1,e_4\in S'\cup U'                                                                                               \\ e_2,e_3\in V(X) \\ e_1\rela e_2, e_2\rela e_3, e_3\rela e_4}} \{e_1,e_2,e_3,e_4\}.
\end{align*}
For an edge $e=(h,v)\in E(M')$, let 
\begin{align*}
\eta_H(e)&=\tau^{-1}(\tau(e))\cap \left(M'_H\times \{v\}\right),\\
\eta_V(e)&=\tau^{-1}(\tau(e))\cap \left(\{h\}\times M'_V\right),
\end{align*}
and these functions act on subset of $E(M')$ element-wise.
Now we are ready to define a few more sets:
\begin{align*}
	Q_H'=&\Big\{ f\in V(X):\ \exists e\in S'\ f\rela e,\ f\cap e\in M'_H,\  \eta_H(f)\setminus \{f\}\covers N_X(\eta_H(e))\setminus \{f\} \Big\}, \\
	Q_V'=&\Big\{ f\in V(X):\ \exists e\in S'\ f\rela e,\ f\cap e\in M'_V,\  \eta_V(f)\setminus \{f\}\covers N_X(\eta_V(e))\setminus \{f\} \Big\}, \\
	Q'=&Q'_H\bigcup Q'_V.
\end{align*}
{\color{highlightnew} The reader is advised to look at \Fref{fig:phase1} to visualize the corresponding pixels.} Take
\begin{align*}
	Z'_1&=\eta_H(C'\cup Q')\bigcup \eta_V(C'\cup Q'),                                                                                                                         \\
	X_1    & =X -Q'-N_X(Q') -U'-N_X(U'),                                                                                                                                              \\
	w_1    & =w_0-\mathds{1}_{S'}-\mathds{1}_{U'}+\sum_{f\in Q'}\mathds{1}_{N_X(N_X(f))\setminus \{f\}\setminus W'}+\sum_{e\in U'}\mathds{1}_{N_X(N_X(e))\setminus\{e\}\setminus W'}.
\end{align*}

\medskip

\phase{}\label{phase:cycle} Suppose $e_1,e_2,e_3,\ldots,e_{2k_j}$ is a cycle in $X_j$ ($k_j\ge 2$, $j\ge 1$). This set of nodes of $X_j$ is the edge set of a \textbf{circuit} of length $2k_j$ in $M'$. Join the centroids of the pixels of edges that are in the $\leftrightarrow$ relation. Observe, that because we split the elements of $B_H$ horizontally and the elements of $B_V$ vertically, this curve either always turns left, or always turns right. Without loss of generality, we may suppose that $e_1,e_2,\ldots,e_{2l}$ is the shortest cycle in $M'$ formed by an interval of edges of the circuit (in cyclic order).
\begin{itemize}
	\item If $l=2$, then $e_1\leftrightarrow e_2\leftrightarrow e_3\leftrightarrow e_4$. Take 
	\begin{align*}
	Z'_{j+1} & =\{e_{2k_j}\}\bigcup Z'_j,     \\
	X_{j+1}  & =X_j-\{e_{2k_j-1},e_{2k_j},e_1,e_2,e_3\}, \\
	w_{j+1}  & = \left\{\begin{array}{ll}
	w_j & \text{ if }2k_j=4, \\
	w_j+\mathds{1}_{e_4}+\mathds{1}_{e_{2k_j-2}} & \text{ if }2k_j>4.
	\end{array}\right.
	\end{align*}
	By definition, $e_{2k_j}\rightarrow e_{2k_j-1},e_1$. Because the curve connecting the centroids always turns in the same direction, we have $e_{2k_j}\rightarrow e_2,e_3$, too.
	\item If $l>2$, the chordal bipartiteness of $M'$ implies that there exists $1<b<2l$ and a chord $f$ which forms a cycle with $e_{b-1},e_b,e_{b+1}$. Check that $f\rightarrow e_b$ holds.
	
	\medskip
	
	If the pixel of $f$ is between the pixels of $e_{b-1}$ and $e_{b-2}$ and is also between the pixels of $e_{b+1}$ and $e_{b+2}$, then $f\rightarrow e_{b-2},e_{b-1},e_{b+1},e_{b+2}$, so we may take
	\begin{align*}
	Z'_{j+1} & =\{f\}\bigcup Z'_j,     \\
	X_{j+1}  & =X_j-\{e_{b-1},e_{b-1},e_b,e_{b+1},e_{b+2}\}, \\
	w_{j+1}  & =w_j+\mathds{1}_{e_{b-3}}+\mathds{1}_{e_{b+3}},
	\end{align*}
	where the indices are taken cyclically.
	
	\medskip
	
	Otherwise $f$ is not, say, between the pixels of $e_{b-1}$ and $e_{b-2}$. Let $$g=(e_{b-2}\setminus e_{b-1})\bigcup (e_b\cap e_{b+1}).$$
	Indeed, $g$ is a chord, since the pixel of $e_{b-2}$ is between the pixels of $e_{b-1}$ and $f$. Now the pixel of $g$ is between the pixels of $e_{b-3}$ and $e_{b-2}$, and also between the pixels of $e_{b}$ and $e_{b+1}$. The argument of the previous case now goes through, so take
	\begin{align*}
	Z'_{j+1} & =\{g\}\bigcup Z'_j,     \\
	X_{j+1}  & =X_j-\{e_{b-3},e_{b-2},e_{b-1},e_b,e_{b+1}\}, \\
	w_{j+1}  & =w_j+\mathds{1}_{e_{b-4}}+\mathds{1}_{e_{b+2}}.
	\end{align*}
	(Again, indices of $e$ are taken cyclically.)
\end{itemize}

By iterating the above operation, eventually we reach an index $j_1$ for which $X_{j_1}$ is cycle-free ($|V(X_j)|$ decreases with every iteration).

\medskip

\phase{}\label{phase:selfintersect} This phase is very similar to the part of Phase~2 corresponding to $l\ge 2$. Take a path component $e_1,e_2,e_3,\ldots,e_{k_j}$ in $X_j$ ($k_j\ge 1$, $j\ge j_1$) such that
\[ E\left(M'\left[\bigcup_{i=2}^{k-1} e_i\right]\right)\setminus \{e_2,\ldots e_{k-1}\}\neq\emptyset.\]
The set $\{e_1,\ldots,e_{k_j}\}$ is the edge set of a \textbf{walk} of length $k_j$ in $M'$. Join the centroids of the pixels of these edges that are in the $\leftrightarrow$ relation. Again, this curve either always turns left, or always turns right. Using the bichordality of $M'$, there exists a chord $f\in E(M')$ which forms a $C_4$ with $\{e_{b-1},e_b,e_{b+1}\}$, where $3\le b\le k-2$.

\medskip

As before, depending on where the pixel of $f$ is, the new element added to $Z'_{j}$ is either $f$ or another chord $g$. If the pixel of $f$ is between the pixels of $e_{b-1}$ and $e_{b-2}$ and is also between the pixels of $e_{b+1}$ and $e_{b+2}$, then we may take
\begin{align*}
Z'_{j+1} & =\{f\}\bigcup Z'_j,                                   \\
X_{j+1}  & =X_j-\{e_{b-2},e_{b-1},e_{b},e_{b+1},e_{b+2}\},       \\
w_{j+1}  & =w_j+\mathds{1}_{\{\mathrm{dist}_X(\bullet,e_b)=3\}}.
\end{align*}

\medskip

Otherwise $f$ is not, say, between the pixels of $e_{b-1}$ and $e_{b-2}$ (so $e_{b-2}$ is a reflex edge thus $b\ge 4$). Let $$g=(e_{b-2}\setminus e_{b-1})\bigcup (e_b\cap e_{b+1}).$$
Indeed, $g$ is a chord, since the pixel of $e_{b-2}$ is between the pixels of $e_{b-1}$ and $f$. Now the pixel of $g$ is between the pixels of $e_{b-3}$ and $e_{b-2}$, and is also between the pixels of $e_{b}$ and $e_{b+1}$. The argument of the previous case now goes through, so take
\begin{align*}
Z'_{j+1} & =\{g\}\bigcup Z'_j,                                   \\
X_{j+1}  & =X_j-\{e_{b-3},e_{b-2},e_{b-1},e_{b},e_{b+1}\},       \\
w_{j+1}  & =w_j+\mathds{1}_{\{\mathrm{dist}_X(\bullet,e_{b-1})=3\}}.
\end{align*}
Since the number of nodes in $X_j$ decreases with every iteration of this method, there is a $j_2$ for which $X_{j_2}$ becomes free of the above defined paths.

\medskip

\phase{}\label{phase:3rdproperty} The set $M'_H$ is the subset of the nodes of a horizontal $R$-tree of $D$. Let $h_\text{root}\in M'_H$ be a horizontal slice whose top side has maximal $y$-coordinate (so only convex and side edges are incident to it in $M'$). Process the elements of $M'_H$ in decreasing distance (measured in the horizontal $R$-tree) from $h_\text{root}$.

\medskip

Let $h_3\in M'_H$ is the next horizontal slice to be processed. Let $h_4\in M'_H$ be the neighbor of $h_3$ on the path between $h_3$ and $h_\text{root}$. Because $M'$ is 2-connected, the path induced by $N_{M'}(h_3)\bigcap N_{M'}(h_4)$ contains at least two nodes; let the end-points of the path be $v_3$ and $v_4$. As it is shown in \Fref{lemma:zprime}, in this case the edges of the cycle $\{v_3,h_3,v_4,h_4\}$ are non-internal edges of $M'$. If not each of them is a reflex edge, continue this phase with the next horizontal slice. Suppose now, that all four edges of the cycle are reflex edges of $M'$.

\medskip

If $\{v_3,h_3\}$ and $\{v_4,h_3\}$, or $\{v_3,h_4\}$ and $\{v_4,h_4\}$ are removed in \Fref{phase:cycle}~or~\Fref{phase:selfintersect} in one iteration, then the edge by which $Z'$ is extended in the same step satisfies the~\ref{item:internal}\textsuperscript{rd} property of hyperguards for $Z'$ and $h_3,h_4$, and we skip to the next horizontal slice to be processed.
{\color{highlightnew} It is also quite possible, however, that $\{v_3,h_3\}$ and $\{v_4,h_3\}$ are removed in different iterations of the previous phases; this case, among others, is handled in the following paragraphs.}

\medskip

If $\{\{v_3,h_3\},\{v_4,h_3\}\}\cap V(X_j)$ is non-empty, take the path component $P_j$ of $X_j$ containing this set; otherwise let $P_j$ be the empty graph. Because of \Fref{phase:selfintersect}, the path traced out by connecting the centroids of the pixels corresponding to the nodes of $P_j$ is without self-intersection. This implies that for any node $e\in V(P)$, its horizontal slice $e\cap M_H$ is at least as far away from the root as $h_3$. {\color{highlightnew} See \Fref{fig:phase4}, for example.}

\medskip

Split the path $P_j$ into two components $P_{j,1}$ and $P_{j,2}$ by deleting $\left\{\{v_3,h_3\},\{v_4,h_3\}\right\}$ (if it it is not in $E(P_j)$, then one of the components is empty, and the other is $P_j$), so that $\{v_3,h_3\}\notin V(P_{j,2})$ and $\{v_4,h_3\}\notin V(P_{j,1})$.

\begin{itemize}
	\item
	      If $|V(P_{j,1})|\not\equiv 0\pmod 3$ or $|V(P_{j,2})|\not\equiv 0\pmod 3$, then let $Y_j$ be a minimum size dominating set of $P_j$ containing $\{v_3,h_3\}$ or $\{v_4,h_3\}$ (the size of $Y_j$ is estimated in \Fref{sec:estimating}). Set
	      \begin{align*}
		      Z'_{j+1}   & =Y_j\bigcup Z'_j, \\
		      X_{j+1}    & =X_j-P_j,         \\
		      w_{j+1}(e) & =\left\{
		      \begin{array}{ll}
			      0,\quad     & \text{if }e\in V(P_j),    \\
			      w_j(e)\quad & \text{if }e\notin V(P_j).
		      \end{array}
		      \right.
	      \end{align*}
	      Clearly, one of $\{v_3,h_3\}$ and $\{v_4,h_3\}$ is contained in $Y_j\subset Z'_{j+1}\subseteq Z'$, and it satisfies the~\ref{item:internal}\textsuperscript{rd} property of hyperguards for $Z'$ and $h_3,h_4$.
	\item
	      If $|V(P_{j,1})|\equiv |V(P_{j,2})|\equiv 0\pmod 3$, then let $Y_j$ be a minimal dominating set of $P_j$. Moreover, if $\{\{v_3,h_4\},\{v_4,h_4\}\}\bigcap (V(X_j)\bigcup Z'_j)$ is non-empty, let $f_j$ be an element of it, otherwise set $f_j=\{v_3,h_4\}$.
	      Take
	      \begin{align*}
		      Z'_{j+1}   & =Y_j\bigcup \{f_j\}\bigcup Z'_j,   \\
		      X_{j+1}    & =X_j-P_j-\{f_j\}-N_{X_j}(\{f_j\}), \\
		      w_{j+1}(e) & =\left\{
		      \begin{array}{ll}
			      0,\quad        & \text{if }e\in V(P_j)\bigcup \big\{\{v_3,h_4\},\{v_4,h_4\}\big\}, \\
			      w_j(e)+1,\quad & \text{if }\mathrm{dist}_{X}(e,f_j)=2,                             \\
			      w_j(e)\quad    & \text{otherwise.}
		      \end{array}
		      \right.
	      \end{align*}
	      Observe, that $f_j$ satisfies the~\ref{item:internal}\textsuperscript{rd} property of hyperguards for $Z'$ and $h_3,h_4$.
\end{itemize}

In any case, some element of $Z'_{j+1}\subseteq Z'$ satisfies the~\ref{item:internal}\textsuperscript{rd} property of hyperguards for $Z'$ and $h_3,h_4$.

\medskip

\phase{}\label{phase:path} Lastly, we get $X_{j_3}$ which is the disjoint union of paths and isolated nodes (or it is an empty graph). Take a component $P_j$ of $X_j$ (for some $j\ge j_3$). Let $Y_j$ be a dominating set of $P_j$ (if $|V(P_j)|=1$, then $Y_j=V(P_j)$). Take
\begin{align*}
	Z'_{j+1}   & =Y_j\bigcup Z'_j, \\
	X_{j+1}    & =X_j-P_j,         \\
	w_{j+1}(e) & =\left\{
	\begin{array}{ll}
		0,\quad     & \text{if }e\in V(P_j),    \\
		w_j(e)\quad & \text{if }e\notin V(P_j).
	\end{array}
	\right.
\end{align*}

\medskip

By repeating this procedure, eventually $X_{j_4}$ is the empty graph for some $j_4\ge j_3$.

\medskip

Let $Z'=Z'_{j_4}$. This whole procedure is orchestrated in a way to guarantee that $Z'$ is a hyperguard of $M'$, so only an upper estimate on the cardinality of $\tau(Z')$ needs to be calculated to complete the proof of \Fref{case:2conn}.

\subsubsection{Estimating the size of \texorpdfstring{\boldmath $Z=\tau(Z')$}{𝑍=𝜏(𝑍')}.}\label{sec:estimating}
We have
\begin{align*}
	|V(X_0)|&=r'+s',\\
	w_0(X)&=s', \\
	|B'_H|+|B'_V|&=|U'|+|Q'_H|+|Q'_V|.
\end{align*}

By definition, 
\[|Z_1'|=c'+|U'|+2|Q'_H|+2|Q'_V|-|Q'_H\cap Q'_V|=c'+|B'_H|+|B'_V|+|Q'|,\]
\[\tau(Z'_1)|\le |Z'_1|-|B'_H|-|B'_V|.\]
Again, via the definitions one gets, a bit tediously, but relatively simply that:
\begin{align*}
|V(X_1)|+w_1(X)+2|U'|+5|Q'\setminus Q'_H|+5|Q'\setminus Q'_V|+9|Q'_H\cap Q'_V|\le |V(X_0)|+w_0(X) \\
|V(X_1)|+w_1(X)+2|U'|+5|Q'|+4|Q'_H\cap Q'_V|\le |V(X_0)|+w_0(X) \\
|V(X_1)|+w_1(X)+2|B'_H|+2|B'_V|+3|Q'|+2|Q'_H\cap Q'_V|\le |V(X_0)|+w_0(X)
\end{align*}
We gain the coefficient 9 because each element of $Q'_H\cap Q'_V$ covers a path component of 3 nodes and 4 other nodes of $X$ (however, even a coefficient of 7 is sufficient to complete the proof).

\medskip

All in all, we have
\begin{align}
	|Z'_1| & + \frac{|V(X_1)|+w_1(X)}{3}\le \nonumber \\
	&\le c'+|B'_H|+|B'_V|+|Q'|+\frac{|V(X_1)|+w_1(X)}{3}\le \nonumber \\
	& \le c'+\frac{|V(X_0)|+ w_0(X)+|B'_H|+|B'_V|}{3}\le \label{ineq:Z1}             \\
	& \le c'+\frac{r'+2s'+|B'_H|+|B'_V|}{3}.\nonumber
\end{align}

We now show that
\begin{align}
	|Z'_{j+1}| & +\frac{|V(X_{j+1})|+w_{j+1}(X)}{3}\le
	|Z'_{j}|+\frac{|V(X_{j})|+w_{j}(X)}{3}.\label{ineq:recursion}
\end{align}
holds for any $j\ge 1$.

\medskip

In \Fref{phase:cycle}, we choose a node from each cycle of $X_1$. \Fref{ineq:recursion} is preserved, since
\begin{align*}
	|Z'_{j+1}|   & =|Z'_j|+1,                                  \\
	|V(X_{j+1})| & =|V(X_j)|-5+\mathds{1}_{\{4\}}(k_j),        \\
	w_{j+1}(X)   & \le w_j(X)+2-2\cdot\mathds{1}_{\{4\}}(k_j).
\end{align*}

In \Fref{phase:selfintersect}, for every $j_2>j\ge j_1$, we have
\begin{align*}
	|Z'_{j+1}|   & =|Z'_j|+1,       \\
	|V(X_{j+1})| & =|V(X_{j_1})|-5, \\
	w_{j+1}(X)   & \le w_{j}(X)+2.
\end{align*}

Next, we analyze \Fref{phase:3rdproperty}. Let $j_3>j\ge j_2$. If $|V(P_{j,1})|\not\equiv 0\pmod 3$ and $|V(P_{j,2})|\not\equiv 2\pmod 3$, then take a minimum size dominating set of $P_j$ containing $\{v_1,h_1\}$. Using \Fref{claim:pathdom}, we have
\begin{align*}
	|Y_j| & \le 1+\left\lceil\frac{|V(P_{j,1})|-2}{3}\right\rceil+\left\lceil\frac{|V(P_{j,2})|-1}{3}\right\rceil \le \\
	      & \le 1+\frac{|V(P_{j,1})|-1}{3}+\frac{|V(P_{j,2})|}{3}= \frac{|V(P_j)|+2}{3}.
\end{align*}
Similarly, if $|V(P_{j,1})|\not\equiv 2\pmod 3$ and $|V(P_{j,2})|\not\equiv 0\pmod 3$, then there is a small dominating set of $P_j$ containing $\{h_1,v_2\}$. Also, if both $|V(P_{j,1})|\equiv 2\pmod 3$ and $|V(P_{j,2})|\equiv 2\pmod 3$ hold, then there is a small dominating set of $P_j$ containing $\{h_1,v_2\}$.
Thus, if $|V(P_{j,1})|\not\equiv 0\pmod 3$ or $|V(P_{j,2})|\not\equiv 0\pmod 3$, then
\begin{align*}
	|Z'_{j+1}|   & =|Z'_j|+|Y_j|\le |Z'_j|+\frac{|V(P_j)|+2}{3}, \\
	|V(X_{j+1})| & =|V(X_{j_1})|-|V(P_j)|,                       \\
	w_{j+1}(X)   & \le w_{j}(X)-2.
\end{align*}

\medskip

If both $|V(P_{j,1})|\equiv 0\pmod 3$ and $|V(P_{j,2})|\equiv 0\pmod 3$, then
$|Y_j|=\frac{|V(P_j)|}{3}$. Observe, that
\[ \{h_1,v_1\},\{h_1,v_2\},\{h_2,v_1\},\{h_2,v_2\}\notin V(P_k)\text{ for any }k<j. \]
If both $\{h_1,v_1\}\notin Z'_j$ and $\{h_1,v_2\}\notin Z'_j$, but were removed in different iterations, then when $\{h_1,v_1\}$ is removed in iteration $k$ we must have set $w_k(\{h_1,v_2\})=1$, which is the consequence of the previous observation. Thus, $w_j(\{h_1,v_2\})=1$.
Similarly, we must have $w_j(\{h_1,v_1\})=1$.
This reasoning holds for $\{h_2,v_1\}$ and $\{h_2,v_2\}$, as well.

\medskip

If $P_j$ is not the empty graph or $f_j\in Z(X_j)$, then \fref{ineq:recursion} trivially holds.
If $P_j$ is the empty graph, then $w_j(\{h_1,v_1\})=w_j(\{h_1,v_2\})=1$. If $f_j\in V(X_j)$, these 2 extra weights can be used to compensate for the new degree 1 vertices of $X_{j+1}$. If $f_j\notin Z(X_j)\bigcup V(X_j)$, then even $w_j(\{h_2,v_1\})=w_j(\{h_2,v_2\})=1$, and in total the 4 extra weights compensate for adding $f_j$ to $Z'_{j+1}$.

\medskip

In any case, \fref{ineq:recursion} holds for $j_3>j\ge j_2$.

\medskip

For any $j_4>j\ge j_3$, we have
\[ |Y_j|\le \left\lceil\frac{|V(P_j)|}{3}\right\rceil\le \frac{|V(P_j)|+2}{3} \] and $w_j(P_j)=2$, so \fref{ineq:recursion} holds for $j$.

\subsubsection{Summing it all up.}
By definition, we have
\[ |Z'|=|Z'_{j_4}|,\ X_{j_4}=\emptyset,\ 0\le w_{j_4}(X). \]
\Fref{ineq:recursion} is preserved from \Fref{phase:cycle} up to \Fref{phase:path}, therefore
\[ |Z'|\le |Z'_{j_4}|+\frac{|V(X_{j_4})|+w_{j_4}(X)}{3}\le |Z'_1|+\frac{|V(X_1)|+w_1(X)}{3}. \]

Lastly, using \fref{ineq:Z1}, we get
\begin{align*}
	|Z|&=|\tau(Z')|=|\tau(Z'\setminus Z'_1)|+|\tau(Z'_1)|\le |Z'\setminus Z'_1|+|Z'_1|-|B'_H|-|B'_V|= \\
	 & =|Z'|-|B'_H|-|B'_V|\le c'+\frac{r'+2s'-2|B'_H|-2|B'_V|}{3}=\\
	 &=c'+\frac{(c'-4)+2s'-2|B'_H|-2|B'_V|}{3}=              \\
	 &=\frac{4\left(c'+\tfrac12 s'\right)-4-2|B'_H|-2|B'_V|}{3}=\frac{4|V(M')|-4-2|B'_H|-2|B'_V|}{3}=           \\
	 & =\frac{4|M'_H|+4|M'_V|-4-4|B_H|-4|B_V|}{3}=\frac{4(|M_H|+|M_V|)-4}{3},
\end{align*}
as desired.

\subsection{\texorpdfstring{\boldmath $M$}{𝑀} is connected, but not 2-connected}\label{case:conn}

Let the 2-connected components (or blocks) of $M$ be $M_i$ for $i=1,\ldots,q$. Since induced graphs of $G$ inherit the chordal bipartite property, by~\Fref{case:2conn}, there exists a subset $Z_i\subseteq E(M_i)$, such that for any edge $e_0\in E(G[N_G(M_i)])$, there exists an edge $e_1\in Z_i$ which has $r$-vision of $e_0$ in $G[N(M_i)]$, and $|Z_i|\le \frac43 (|V(M_i)|-1)$. Let $Z=\cup_{i=1}^q Z_i$.

\medskip

Since the intersection graph of the vertex sets of the 2-connected components is a tree (and any two components intersect in zero or one elements), we have
\[ |Z|\le \frac43 \left(-q + \sum_{i=1}^q |V(M_i)| \right)= \frac{4\left(-q + |V(M)| + (q-1)\right)}{3}=\frac{4(|V(M)|-1)}{3}. \]

\medskip

Furthermore, given an arbitrary $e_0=\{v_0,h_0\}\in E(G)$, there exists a $v_1\in M_V$ and an $h_1\in M_H$ such that $\{v_1,h_0\},\{v_0,h_1\}\in E(G)$.
\begin{itemize}
	\item If $v_0\in M_V$ or $h_0\in M_H$, then $\{v_0,h_1\}$ or $\{v_1,h_0\}$ is in $E(M)$.

	\item Otherwise, there exists a path in $M$ whose endpoints are $v_1$ and $h_1$, and this path and the edges $\{v_1,h_0\}$,$\{h_0,v_0\}$,$\{v_0,h_1\}$ form a cycle in $G$.
	      By the bichordality of $G$, there exists a $C_4$ in $G$ which contains an edge of $M$ and $e_0$.
\end{itemize}
In any case, $e_0$ is $r$-visible from some $e_1\in E(M)$. As $e_1$ is an edge of one of the 2-connected components $M_i$, we have $e_0\subset N_G(M_i)$, therefore $e_0\in E(G[N_G(M_i)])$. Thus, some $e_2\in Z_i$ has $r$-vision of $e_0$.

\subsection{\texorpdfstring{\boldmath $M$}{𝑀} has more than one connected component.}\label{case:disconnected}

Let us take a decomposition of $M$ into connected components $M_i$ for $i=1,\ldots,t$.

\medskip

Let $N_i=N(M_i)$, so we have $M_i\subseteq N_i$ and $\cup_{i=1}^t N_i=V(G)$.

\medskip

For all $i>1$ let $q_i$ be the number of components of $G[\cup_{k=1}^{i-1} N_k\setminus\cup_{k=i}^{t}N_k]$ to which $N_i\setminus\cup_{k=i+1}^{t}N_k$ is joined in $G[\cup_{k=1}^i N_k\setminus\cup_{k=i+1}^{t}N_k]$.
Let $F_{i,j}$ be the set of edges joining $N_i\setminus\cup_{k=i+1}^{t}N_k$ to the $j^{th}$ component of $G[\cup_{k=1}^{i-1} N_k\setminus\cup_{k=i}^{t}N_k]$.
Furthermore, let $F_{i,j}^V=\{f\in F_{i,j}\ |\ f\cap A_V\cap N_i\neq \emptyset\}$ and $F_{i,j}^H=\{f\in F_{i,j}\ |\ f\cap A_H\cap N_i\neq \emptyset\}$.

\begin{claim}\label{claim:f1}
	For any two edges $f_1,f_2\in F_{i,j}^V$ either $f_1\cap f_2\neq\emptyset$ or $\exists f_3\in F_{i,j}^V$ such that $f_3$ intersects both $f_1$ and $f_2$. The analogous statement holds for $F_{i,j}^H$.
\end{claim}
\begin{proof}
	Suppose $f_1$ and $f_2$ are disjoint. Since $M_i$ is connected, there is a path in $G$ whose endpoints are $f_1\cap N_i$ and $f_2\cap N_i$, while its internal points are in $V(M_i)$; let the shortest such path be $Q_1$. There is also a path in the $j^{th}$ component of $G[\cup_{k=1}^{i-1}N_k\setminus \cup_{k=i}^t N_k]$ whose endpoints are $f_1\setminus N_1$ and $f_2\setminus N_i$, let the shortest one be $Q_2$.

	\medskip

	Now $Q_1,f_1,Q_2,f_2$ form a cycle in $G[\cup_{k=1}^i N_k\setminus \cup_{k=i+1}^t N_k]$, which is bipartite chordal. Since $V(Q_2)\cap N_i=\emptyset$, there cannot be a chord between $V(M_i)\cap V(Q_1)$ and $V(Q_2)$. This implies that $|V(Q_1)|=3$ by its choice, and that either $(f_1\cap N_i)\cup(f_2\setminus N_i)$ or $(f_2\cap N_i)\cup (f_1\setminus N_i)$ is a chord.
\end{proof}

\begin{claim}\label{claim:f2}
	For any two edges $f^V\in F_{i,j}^V$ and $f^H\in F_{i,j}^H$, the two-element set
	\[ (f^V\cap N_i)\cup (f^H\cap N_i)\]
	is an edge of $G[N_i]$.
\end{claim}
\begin{proof}
	Similar to the proof of \Fref{claim:f1}.
\end{proof}

Let $f_{i,j}^V\in F_{i,j}^V$ be the element which intersects the maximum number of edges from $F_{i,j}$, and choose $f_{i,j}^H\in F_{i,j}^H$ in the same way. If only one of these exist, let $w_{i,j}$ be the existing one, otherwise let $w_{i,j}=(f_{i,j}^V\cap N_i)\cup (f_{i,j}^H\cap N_i)$ (as in \Fref{claim:f2}). Let us finally define
\[ W=\{w_{i,j}\ |\ i=2,\ldots,t\text{ and }j=1,\ldots,q_i\}.\]

\begin{claim}\label{claim:Wsize}
	$|W|=t-1$.
\end{claim}
\begin{proof}
	Observe that for every $i=1,\ldots,t$, the subgraph $G[N_i\setminus\cup_{k=i+1}^{t}N_k]$ is connected, since $M_i\subseteq N_i\setminus\cup_{k=i+1}^{t}N_k\subseteq N_i=N(M_i)$. Moreover, $G[\cup_{k=1}^{t}N_k]=G$ is connected, therefore $t-1=\sum_{i=2}^t q_i=|W|$.
\end{proof}

By~\Fref{case:conn}, there exists a subset $Z_i\subseteq E(M_i)$, such that for any edge $e_0\in E(G[N_i])$ there exists an edge $e_1\in Z_i$ which has $r$-vision of $e_0$ in $G[N_i]$, and $|Z_i|\le \frac43 (|V(M_i)|-1)$.

Let $Z=W\cup \left(\cup_{i=1}^t Z_i\right)$. An easy calculation gives that
\begin{align*}
	|Z| & \le (t-1)+\sum_{i=1}^t \frac{4|V(M_i)|-4}{3}\le \frac{4|V(M)|-4t+3(t-1)}{3}\le \\
	    & \le \frac{4(|M_H|+|M_V|-1)}{3}.
\end{align*}

Take an arbitrary edge $e_0=\{v_0,h_0\}\in E(G)$. We have three cases.

\begin{enumerate}
	\item
	      If $e_0\in F_{i,j}^V$ for some $i,j$, then we claim that $f_{i,j}^V$ has $r$-vision of $e_0$. Suppose not; then $f_{i,j}^V\cap e_0=\emptyset$, and $f_1:=\{v_0\}\cup (f_{i,j}^V\setminus N_i)\notin E(G)$ or $f_2:=\{h_0\}\cup (f_{i,j}^V\cap N_i)\notin E(G)$. By \Fref{claim:f1} at least one of them is in $E(G)$. Suppose $f_1\in E(G)$ and $f_2\notin E(G)$. For any edge $e\in F_{i,j}^V$ intersecting $f_{i,j}^V\cap N_i$, there is an edge $f(e)\in E(G)$ which intersects both $e$ and $e_0$. As $f(e)\neq f_2$, we must have $f(e)=(f_{i,j}^V\cap N_i)\cup (e\setminus N_i)$. Furthermore, any edge $g\in F_{i,j}^V$ intersecting $f_{i,j}^V\setminus N_i$ is trivially intersected by $f_1$ also.
	      Thus, $f_1$ intersects at least as many edges as $f_{i,j}^V$, and $f_1$ intersects $e_0$ too, which contradicts the choice of $f_{i,j}^V$. By symmetry, we are also done if $f_1\notin E(G)$ and $f_2\in E(G)$.

	      \medskip

	      If $w_i=f_{i,j}^V$, then $w_i$ trivially has $r$-vision of $e_0$. If both $f_{i,j}^V$ and $f_{i,j}^H$ exist, we have two cases.
	      \begin{itemize}
		      \item If $v_0\in f_{i,j}^V$, then $v_0\in w_i$ too, so $w_i$ has $r$-vision of $e_0$.

		      \item Otherwise, \Fref{claim:f2} yields that $\{v_0\}\cup (f_{i,j}^H\cap N_i)\in E(G)$. Also, $f_{i,j}^V$ has $r$-vision of $e_0$, so $\Big\{f_{i,j}^H\cap N_i,v_0,h_0,f_{i,j}^V\cap N_i\Big\}$ is the vertex set of a $C_4$ in $G$, so $w_i$ has $r$-vision of $e_0$.
	      \end{itemize}

	\item If $e_0\in F_{i,j}^H$ for some $i,j$, the same argument as above gives that $w_{i,j}$ has $r$-vision of $e_0$.

	\item If neither of the previous two cases holds, then $e_0\in E(G[N_i])$ for some $i$, so some element of $Z_i$ has $r$-vision of it.
\end{enumerate}

Thus, $Z$ satisfies \Fref{thm:mainprime}, and the proof is complete.

\section{Algorithmic aspects}\label{sec:algo}

Finding a minimum cardinality horizontal mobile $r$-guard system, which is also known as the \textsc{Minimum cardinality Horizontal Sliding Cameras} or \textsc{MHSC} problem,
is known to be polynomial~\cite{KM11} in orthogonal polygons without holes. In orthogonal polygons with holes, the problem is \textsc{NP}-hard as shown by \citet{BCLMMV16}. In their paper, a polynomial time constant factor approximation algorithm for the \textsc{MHSC} problem is described, too.
As explained in \Fref{sec:translating}, the \textsc{MHSC} problem translates to the \textsc{Total Dominating Set} problem in the pixelation graph (\Fref{sec:translating}), which can be solved in polynomial time for chordal bipartite graphs~\cite{DMK90}.

\medskip

\textcolor{highlightnew}{The minimum cardinality weakly cooperative mobile guard set problem in two-dimensional grids (\textsc{MinWCMG} for short) is \textsc{NP}-complete \cite{MR2310594}. However, \citeauthor{MR2310594} also propose a quadratic time algorithm for \textsc{MinWCMG} in simple grids. This is exactly the same problem to which we reduce our problem in \Fref{sec:translating}.}

\medskip

Finding a minimum cardinality mixed vertical and horizontal
mobile $r$-guard system (also known as the \textsc{Minimum cardinality Sliding Cameras} or \textsc{MSC} problem) has been shown by \citet{DM13} to be \textsc{NP}-hard for orthogonal polygons with holes. For orthogonal polygons without holes, the problem translates to the \textsc{Dominating Set} problem in the pixelation graph. This reduction in itself has little use, as \citet{MB87} have shown that \textsc{Dominating Set} is \textsc{NP}-complete even in chordal bipartite graphs. To our knowledge, the complexity of \textsc{MSC} is still an open question for orthogonal polygons. There is, however, a polynomial time 3-approximation algorithm by \citet{KM11} for the \textsc{MSC} problem for $x$-monotone orthogonal polygons without holes. Also, for an orthogonal polygon of $n$ vertices, a covering set of mobile guards of cardinality at most $\left\lfloor (3n+4)/16 \right\rfloor$ (which is the extremal bound shown by \citet{Ag}) can be found in linear time~\cite{GyM2016}. In case holes are allowed,~\cite{BCLMMV16} give a polynomial time constant factor approximation algorithm.

\medskip

The algorithm for the \textsc{MHSC} problem in~\cite{KM11} relies on a polynomial algorithm solving the \textsc{Clique Cover} problem in chordal graphs. Our analysis of the $R$-tree structures and the pixelation graph allows us to reduce the polynomial running time to linear.

\begin{theorem}\label{thm:mgalg}
	The algorithm in Appendix~\ref{algo:mhsc} finds a solution to the \textsc{MHSC} problem in linear time for simple orthogonal polygons.
\end{theorem}
\begin{proof}
	\citet[Section 5]{GyH} showed that both the horizontal $R$-tree $T_H$ and the vertical $R$-tree $T_V$ of $D$ can be constructed in linear time.

	\medskip

	The main idea of the algorithm is to only sparsely construct the pixelation graph $G$ of $D$. Observe, that the neighborhood of a vertical slice in $G$ is a path in $T_H$, and vice versa. Label each horizontal edge of $D$ by the horizontal slice that contains it. Furthermore, label each vertical edge of each horizontal slice by the edge of $D$ containing it; do this for the horizontal edges of vertical slices as well. This step also takes linear time. The endpoints of a path induced by the neighborhood of any node in $G$ (see \Fref{claim:treeprop}) can be identified via these labels in $O(1)$ time.

	\medskip

	In \Fref{sec:translating}, we showed that a horizontal guard system is a subset of $V(T_H)$ which intersects (covers) each element of $\mathcal{F}_H=\{N_G(v)\ |\ v\in V(T_V)\}$. Dirac's theorem~\cite[p.~10]{frank_dopt} states that $\nu$, the maximum number of disjoint subtrees of the family, is equal to $\tau$, the minimum number of nodes covering each subtree of the family. Obviously, $\nu\le \tau$. The other direction is proved using a greedy algorithm:
	\begin{enumerate}
		\item Choose an arbitrary node $r$ of $T_H$ to serve as its root. The distance of a vertical slice $v\in V(T_V)$ from $r$ is $\mathrm{dist}_r(v)=\min_{h\in N_G(v)}\mathrm{dist}(h,r)$, and let $h_r(v)=\arg\min_{h\in N_G(v)}\mathrm{dist}(h,r)$.
		\item Enumerate the elements of $V(T_V)$ in decreasing order of their distance from $r$, let $v_1,v_2,\ldots,v_{|V(T_V)|}$ be such an indexing. Let $S_0=\emptyset$.
		\item If $N_G(v_i)$ is disjoint from the elements of $\{N_G(v)\ |\ v\in S_{i-1}\}$, let $S_i=S_{i-1} \cup \{ v_i\}$; otherwise let $S_i=S_{i-1}$.
	\end{enumerate}
	We claim that $\{h_r(v)\ |\ v\in S_{|V(T_V)|}\}$ is a cover of $\mathcal{F}_H$. Suppose there exists $v_j\in V(T_V)$ such that $N_G(v_j)$ is not covered. Let $i$ be the smallest index such that $v_i\in S_i$ and $N_G(v_j)\cap N_G(v_i)\neq\emptyset$. Clearly, $i<j$, therefore $\mathrm{dist}_r(v_i)\ge \mathrm{dist}_r(v_j)$. However, this means that $h_r(v_i)\in N_G(v_j)$.

	\medskip

	Now $\{h_r(v)\ |\ v\in S_{|V(T_V)|}\}$ is a cover of the same cardinality as the disjoint set system $\{N_G(v)\ |\ v\in S_{|V(T_V)|}\}$, proving that $\nu=\tau$.

	\medskip

	Each neighborhood $N_G(v)$ for $v\in V(T_V)$ is the vertex set of a path in $T_H$. Therefore, the first part of the algorithm, including calculating $\mathrm{dist}_r(v)$ and $h_r(v)$ for each $v$, can be performed in $O(n)$ time, using the off-line lowest common ancestors algorithm of \citet{MR801823}.

	\medskip

	Calculating the distance decreasing order takes linear time via breadth-first search started from the root. In the $i^{\mathrm{th}}$ step of the third part of the algorithm, we maintain for each node in $V(T_H)$ whether it is under an element of $\{h_r(v)\ |\ v\in S_i\}$. Summed up for the $|V(T_H)|$ steps, this takes only linear time. $N_G(v_{i+1})$ is disjoint from the elements of $\{N_G(v)\ |\ v\in S_i\}$ if and only if one of the ends of the path induced by $N_G(v_{i+1})$ is under one of the elements of $\{h_r(v)\ |\ v\in S_i\}$, which now can be checked in constant time. Thus, the algorithm takes in total some constant factor times the size of the input time to run.
\end{proof}

\textcolor{highlightnew}{If we replace the construction of $R$-trees with sweeps in \Fref{algo:mhsc}, the modified algorithm solves the \textsc{MinWCMG} problem in simple grids. The two sweeps that recover the $R$-trees now dominate the increased time complexity of $O(n\log n)$.}

\medskip

The computational complexity of the \textsc{Point guard} problem in orthogonal polygons with or without holes has attracted significant interest since the inception of the problem. \citet{MR1330860} showed that even for orthogonal polygons (without holes), \textsc{Point guard} is \textsc{NP}-hard. However, a minimum cardinality \textsc{Point $r$-guard} system  of an orthogonal polygon can be computed in $\tilde O(n^{17})$ time~\cite{WM07}. To our knowledge, the exponent of the running time is still in the double digits, which makes its use impractical. Therefore, approximate solutions to the problem are still relevant. A linear-time 3-approximation algorithm is described in~\cite{LWZ12}.

\begin{corollary}
	An $\frac83$-approximation of the minimum size of a point $r$-guard system of a simple orthogonal polygon can be computed in linear time.
\end{corollary}
\begin{proof}
	Compute $m_V$ and $m_H$ using the previous algorithm. By \Fref{thm:main} and the trivial statement that both $m_H\le p$ and $m_V\le p$, we get that $\frac43\cdot(m_H+m_V)$ is an $\frac83$-approximation for $p$.
\end{proof}

Unfortunately, we can only compute the corresponding solution (guard system) in $O(n^2)$, because the pixelation graph may have $\Omega(n^2)$ edges. We consider it an interesting open problem to reduce this running time to linear as well, as such an algorithm would be comparable to the algorithm of~\citet{LWZ12}.

% a lemma 1 transzformációt nem lehet használni partícionálásos tétél bizonyításához

\phantomsection
\renewcommand*{\bibfont}{\small}
\printbibliography

\clearpage

\begin{appendices}
	\section{A linear time algorithm for \textsc{MHSC}}\label{algo:mhsc}
	%\caption{Finding a minimum cardinality horizontal $r$-guard system}
	\begin{algorithmic}[1]
		\Function{Solve \textsc{MHSC}}{P}
		\State $T_H\gets \Call{horizontal $R$-tree}{P}$\Comment{Algorithm~of~\cite[Section 5]{GyH}}
		\State $T_V\gets \Call{vertical $R$-tree}{P}$
		\ForAll{vertical slice $t\in T_V$}
		\State $a,b\gets$ vertical sides of $P$ bounding $t$
		\State $h_a\gets$ horizontal slice in $V(T_H)$ containing $a$
		\State $h_b\gets$ horizontal slice in $V(T_H)$ containing $b$
		\State $\textit{N}[t]\gets \{h_a, h_b\}$
		\EndFor
		\Statex
		\State $r\gets $ arbitrary node of $T_H$ to serve as root
		\State $\textit{dist}[]\gets \Call{Breadth First Search}{T_H, r}$\Comment{distance from $r$}
		\Statex
		\State $\textit{LCA}[]\gets \Call{Lowest Common Ancestors}{T_H, r, \textit{N}[]}$\Comment{Algorithm of~\cite{MR801823}}
		\State \Comment{$\textit{LCA}[t]$ contains the lowest common ancestors of the elements of $\textit{N}[t]$}
		\Statex
		\State $S\gets\emptyset$
		\State Set every node of $T_H$ unmarked
		\ForAll{$t\in V(T_V)$ so that $\textit{dist}[\textit{LCA}[t]]$ is not increasing}\Comment{reverse BFS-order}
		\If{both elements of $\textit{N}[t]$ are unmarked}
		\State $S\gets S\cup \{\textit{LCA}[t]\}$
		\State \Call{Set Mark}{\textit{LCA}[t]}
		\EndIf
		\EndFor
		\State\Return $S$
		\EndFunction

		\Statex

		\Function{Set Mark}{$u$}
		\State mark $u$
		\ForAll{neighbor $w$ of $u$ in $T_H$}
		\If{$\textit{dist}[w]>\textit{dist}[u]$ and $w$ is unmarked}
		\State \Call{Set Mark}{w}
		\EndIf
		\EndFor
		\EndFunction
	\end{algorithmic}
\end{appendices}

\end{document}